\documentclass[a4paper,american,reqno]{amsart}

\newif\ifPreprint \Preprinttrue
\newif\ifSubmission \Submissionfalse

\usepackage{babel}
\usepackage[utf8]{inputenc}
\usepackage[T1]{fontenc}
\usepackage{booktabs}
\usepackage{xspace}
\usepackage[binary-units=true]{siunitx}
\usepackage{etoolbox}
\robustify\bfseries
\usepackage{tabularx}
\usepackage[nohyperlinks,nolist]{acronym}

\usepackage{placeins}
\usepackage{relsize}
\usepackage{algorithm,algpseudocode}
\usepackage{todonotes}
\usepackage{csquotes}
\algrenewcommand\algorithmicrequire{\textbf{Input:}}
\algrenewcommand\algorithmicensure{\textbf{Output:}}
\makeatletter
\renewcommand{\ALG@name}{Method}
\makeatother
\usepackage{microtype}
\usepackage[style = numeric-comp,
            maxbibnames = 100,
            maxcitenames = 2,
            giveninits = true,
            isbn = false,
            backend = biber]{biblatex}
\usepackage[colorlinks,
            citecolor=blue,
            urlcolor=blue,
            linkcolor=blue]{hyperref} 
\usepackage{orcidlink}
\usepackage{cleveref}

\usepackage{caption}
\crefname{algocf}{method}{methods}
\Crefname{algocf}{Method}{Methods}

\usepackage{thmtools,thm-restate}
\declaretheorem{theorem}

\makeatletter
\patchcmd{\@settitle}{\uppercasenonmath\@title}{\scshape\large}{}{}
\patchcmd{\@setauthors}{\MakeUppercase}{\scshape\normalsize}{}{}
\makeatother

\vfuzz \hfuzz
\newcommand{\abbr}[1][abbrev]{#1.\xspace}

\newcommand{\eg}{\abbr[e.g]}

\newcommand{\ie}{\abbr[i.e]}

\newcommand{\define}{\mathrel{{\mathop:}{=}}}
\newcommand{\enifed}{\mathrel{{=}{\mathop:}}}

\newcommand{\field}{\mathbb}
\newcommand{\naturals}{\field{N}}

\newcommand{\reals}{\field{R}}
\newcommand{\integers}{\field{Z}}

\newcommand{\R}{\reals}

\makeatletter
\newcommand{\fcdot}{\,\cdot\,}
\newcommand{\fcarg}[1]{\def\fc@rg{#1}\ifx\fc@rg\empty\fcdot\else\fc@rg\fi}
\makeatother
\newcommand{\abs}[1]{\lvert\fcarg{#1}\rvert}

\newcommand{\norm}[2][]{\lVert\fcarg{#2}\rVert\ifx#1\empty\else_{#1}\fi}
\newcommand{\Norm}[2][]{\left\lVert#2\right\rVert\ifx#1\empty\else_{#1}\fi}

\newcommand{\eps}{\varepsilon}

\newcommand{\domain}{\mathcal{D}}
\newcommand{\locdomain}{\domain_{\text{loc}}}
\newcommand{\lb}[1]{\underline{#1}}
\newcommand{\ub}[1]{\overline{#1}}

\newcommand{\vectorize}[1]{\mathbf{#1}}
\newcommand{\vv}{\vectorize{v}}
\newcommand{\vd}{\vectorize{d}}

\newtheorem{lemma}[theorem]{Lemma}

\newtheorem{remark}[theorem]{Remark}

\crefname{assumption}{assumption}{assumptions}

\newcommand{\software}[1]{\texttt{#1}}

\bibliography{literature}

\usepackage[foot]{amsaddr}

\begin{document}

\title[MINLP Relaxations: Piecewise Linear vs. Global Parabolic]{Clash of MINLP Relaxations: \\ Piecewise Linear vs. Global Parabolic}
\author[A. G{\"o}{$\beta$}]
{Adrian G{\"o}{$\beta$}\orcidlink{https://orcid.org/0009-0002-7144-8657}}

\address[A. Gö$\beta$]{
	University of Technology Nuremberg,
	Analytics \& Optimization,
	Dr.-Luise-Herzberg-Str.~4,
	90461~Nuremberg,
	Germany}
\email[A. Gö$\beta$]{adrian.goess@utn.de}

\date{\today}

\begin{abstract}
  Solving mixed-integer nonlinear programs (MINLPs) 
typically relies on constructing relaxations that are easier to tackle 
than the original problem. 
Recently, global parabolic (PARA) relaxations were introduced, 
featuring separable quadratic functions -- paraboloids -- 
as global under- or overestimators of general nonlinear constraint functions.
So far, the paraboloids are all computed at once by solving
a mixed-integer linear program (MIP).
For small tolerances or wide function domains, the corresponding MIP grows in
size and is eventually intractable,
preventing a meaningful comparison with established relaxation techniques.

We therefore propose a novel iterative method to compute 
PARA approximations that succeeds on all tolerance-domain
combinations where the original one has failed.
The computational study is preceded by a thorough theoretical explanation
and analysis.
Finally, the improved method enables a computational comparison with piecewise
linear (PWL) relaxations in terms of runtime on general MINLP instances.

The results show that the modern solver SCIP can solve PWL relaxations faster
when the tolerance is high, shifting strongly in favor of PARA for tighter
tolerances.
We attribute the effect to the difference in the corresponding problem size:
PWL relaxations introduce binary variables to identify the active linear piece
and their number grows with decreasing tolerance.
PARA, on the other hand, does not require additional variables such that the
dimension is maintained.
For problems with at least one (co)sine constraint, the effect significantly
amplifies.
Thereby, for medium tolerances, PARA relaxations outperform
SCIP stand-alone.
Applied problems like alternating current optimal power flow (AC-OPF) feature
such constraint types, leaving PARA a viable relaxation strategy.

\end{abstract}

\keywords{Mixed-integer Nonlinear Programming,
Relaxations,
Mixed-Integer Linear Programming,
Mixed-Integer Quadratically-Constrained Programming\hspace{-3pt}
}
\subjclass[2010]{90C11, 
90C20, 
90C26, 
90C30\hspace{-3pt} 
}

\maketitle\
\section{Introduction}
\label{sec:intro}

The modeling power of \acf{minlp} is undisputed.
The combination of integral components and nonlinear interconnection allows for
the resemblance of a great variety of applications.
These range from gas network control~\cite{koch2015gasnetwork} over
chemical processes~\cite{baliban2012go-chemical-processes} up to the
design of water networks~\cite{bragalli2012waternetworkdesign}.
Along with this power comes a certain (computational) cost, which makes it hard
to solve general \ac{minlpps} at scale and necessitates a retreat to
simplified models, at least as an intermediate step.
For instance, in modern optimization software for \ac{minlp} like 
\software{BARON}~\cite{sahinidis1996baron},
\software{COUENNE}~\cite{belotti2009branch-and-bound-minlp}, 
\software{Gurobi}~\cite{gurobi},
\software{SCIP}~\cite{hojny2025scip10}, or
\software{XPRESS}~\cite{xpress2025}
relaxations of the original problem are constructed and then solved.

A popular technique to construct relaxations of problems with general nonlinear
constraints is their piecewise handling.
With the advances in \ac{mip} throughout the last decades,
constructing relaxations based on linear functions appears as a potentially
successful option.
In fact, \ac{pwl} relaxations have been originally introduced
in~\cite{markowitz1957orig-intro-pwl} -- to the best of our knowledge --
and are further developed and adjusted every
since, \eg, see~\cite{beach2024nonconvex-miqcp-to-mip,
burlacu2020mip-for-minlp,
croxton2003pwl-comparison,
huchette2023pwl-new-formulations,
rebennack2015pwl-2d}.
Also refer to~\cite{vielma2015pwl-techniques} for an extensive survey and
to~\cite{braun2023pwl-t-rex} for a recent computational comparison.
As fas as we are aware, all of the aforementioned solvers include such
relaxations.

There also exist piecewise techniques beyond the linear realm.
For instance, the SC-\acs{minlp}
approach~\cite{dambrosio2012sc-minlp,dambrosio2019sc-minlp-next-gen}
can be seen as an extension of \ac{pwl}.
For a nonlinear and twice continuously differentiable function $f$, it divides
the function domain into intervals on which $f$ is either convex or concave,
respectively.
Then, the concave parts are relaxed by \ac{pwl}, whereas the convex parts
remain.
This renders the resulting relaxed problem a convex \acs{minlpp}, which is
considered (more) tractable.
Even more general approaches feature
piecewise quadratic~\cite{song2025pwq} or piecewise
polynomial approximates~\cite{cuesta2025additive-regression-go}.
The resulting problems fall into the class of \ac{miqcp}, at least after
reformulation.
Though \ac{miqcp} is still a subclass of \ac{minlp}, there exist tailored
approaches to solve such problems~\cite{beach2022mip-in-qp,
    beach2024nonconvex-miqcp-to-mip,buchheim2013nonconvex-quadratic-integer,
buchheim2013relax-nonconvex-miqp,elloumi2025miqcp-interval-methods,
misener2013glomiqo,misener2014antigone,wiese2021miqcqp-study}.

All those piecewise approaches have the benefit to construct relaxations
that can fulfill a prespecified tolerance.
That is, given a tolerance $\eps > 0$, one can compute approximates, \ie,
pieces of linear/quadratic functions, that
fulfill the $\eps$-accuracy with respect to the original problem.
Thus any solution to the relaxed problem constitutes an $\eps$-approximate
solution which often suffice in practice.
As an additional advantage, the character of the relaxed problem as a \ac{mip}
or \ac{miqcp} problem allows to leverage tailored solvers.

Naturally, those benefits are not without cost.
Depending on the modeling technique, they come with the introduction of binary
variables whose number grows with the number of pieces.
The approximates -- \eg, linear pieces --  lack a global validity, \ie,
they do not constitute global under-/overestimators of the constraint function,
and thus need to be ``disabled'' outside their corresponding interval.
Hence, the resulting relaxation grows in (variable) size, which may
impose computational challenges for large problems or tight tolerances.
In a first attempt to circumvent additional integer variables,
the authors of~\cite{goess2025para} introduced a framework
to compute \ac{para} approximations of functions.
The approximates are \emph{paraboloids}, \ie, separable quadratic functions,
globally under-/overestimating a constraint function.
Although not necessarily being convex, the global validity allows to incorporate
the paraboloids as constraints without the necessity of additional
binary variables.
That is, the problem size grows in the number of constraints only, but some
solvers like SCIP seem to handle those similar-shaped
constraints quite well, compare~\cite{goess2025para}.

Although the \ac{para} approach seems to be a viable relaxation technique,
the method to compute respective \ac{para} approximations
in~\cite{goess2025para} is computationally limited.
It computes all approximating paraboloids at once by repetitively solving
\ac{mip} problems.
Thereby, their size grows with domain width and decreasing tolerance,
leaving them eventually not solvable under limited resources.
Consequently, \cite{goess2025para} reports valid parabolic approximations only
on artificial domains and with respect to moderate tolerances,
preventing a meaningful comparison with established techniques like \ac{pwl}.

We thus introduce a novel iterative method that leverages differentiability of
the constraint functions in combination with a factorable problem structure
to compute \ac{para} approximations for tight tolerances and wide domains.
In \Cref{sec:algo}, we introduce our method formally, give proof of its
correctness, and showcase that it is indeed able to succeed on tolerance-domain
combinations that are left as unsolved by the approach in~\cite{goess2025para}.
Furthermore, our method then allows for a comparison against \ac{pwl}
relaxations in terms of numbers of approximates -- linear pieces versus
parabolas -- (\Cref{subsubsec:number-para-lins}) and computational
efficacy (\Cref{sec:relaxation-comparison}).

The remainder is dedicated to establish common ground in
\Cref{sec:preliminaries} by a presentation of the notation and the optimization
problem.
Further, it gives a short introduction into \ac{pwl} (\Cref{subsec:pwl-model})
and the original \ac{para} approach (\Cref{subsec:para-model}).
We close the article with some concluding remarks in \Cref{sec:conclusion}.
Proofs and additional results can be found in \Cref{sec:appendix-proofs}
and \Cref{sec:dual-gap-improvements}, respectively.

\section{Preliminaries}
\label{sec:preliminaries}

Before diving into the main topic of the article, we want to clarify notation
 and state the problem(s) in question.
During later statements and proofs this will foster clarity.
Additionally, we aim to explain basic concepts such as the relaxation
techniques.

\subsection{Notation}
Consider the naturals without zero $\naturals = \{1, 2, 3,\dots\}$.
We will use $n',\, n \in \naturals$ as the original variable dimension and the
variable dimension after reformulation, respectively, and
$m', \, m \in \naturals$ as the number of original constraints and
constraints of interest, respectively.
For counting or numbering variables, constraints, or other objects,
we introduce the short notation $[n] \define \{1, 2, \dots, n\}$.
Vectors are given in bold such as $\mathbf{x} \in \reals^n$,
whereas components are italicized, e.g., $\mathbf{x} = (x_1, \dots, x_n)^\top$.
For an interval $\domain = [\lb x, \ub x]$, we denote its length by 
$\abs{\domain} \define \abs{\ub x - \lb x}$,
using $\domain \neq \emptyset$ and $\abs{\domain} > 0$ interchangeably.
The interior is denoted by $\mathrm{int}(\domain) = (\lb x, \ub x)$.

\subsection{Mathematical Problem}
\label{subsec:math-problem}

We consider general \ac{minlp} problems of the form
\begin{equation}
	\begin{aligned}
		\min_{\mathbf{x'} \in \reals^{n'}}\ & c(\mathbf{x'}) \\
		\text{s.t.} \ & g_j(\mathbf{x'}) \leq 0, & j \in [m'], \\
		& x'_i \in \integers, & i \in I,\\
        & \mathbf{x'} \in [\lb{\mathbf{x}}',\ub{\mathbf{x}}'],
	\end{aligned}
	\tag{P$_\text{gen}$}
	\label{eq:problem-general}
\end{equation}
where $I \subseteq [n']$.
The vectors $\lb{\mathbf{x}}'$ and $\ub{\mathbf{x}}'$ denote
component-wise bounds to $\mathbf{x}'$, which we assume to be finite.
Note that this kind of formulation is no restriction, \ie, it does include
 equality constraints by re-modeling $g(\mathbf{x'}) = 0$ to 
 $g(\mathbf{x'}) \leq 0$ and $-g(\mathbf{x'}) \leq 0$.

The involved functions are assumed to be sufficiently smooth,
which is common in \ac{minlp}, see,
\eg,~\cite{dambrosio2020-sc-minlp-disjunctive-cuts}.
That is, we assume them to be once continuously differentiable and, if
necessary, to be able to compute their optimum over any compact domain.
Depending on the particular method the latter can require twice continuous
differentiability.

We also assume the involved functions to be
\emph{factorable}, see, \eg, \cite{belotti2013minlo,mccormick1976factorable}.
That is, they are composed of finitely many binary operations (sum, product,
division, power) and finitely many unary operations ($\sin$, $\cos$, $\exp$,
$\log$, $\abs{\cdot}$) of constants and variables.
Equivalently, factorable functions can be expressed in terms
of an expression tree which has the mentioned operations as nodes
as well as variables and constants as leaves.
For a given non-leaf node $N$, the child node(s) represents the argument(s) of
the operation in $N$.
More detailed explanations can be found
in~\cite{cohen2002elementary-algorithms,belotti2013minlo,
gay2012expressiongraphs,morsi2013diss,schichl2005expressiontrees}.
Nearly all instances from MINLPLib~\cite{bussieck2003minlplib} are available in
the \software{OSiL} file format~\cite{fourer2010osil},
which directly provides the expression tree structure.

The factorability allows to rewrite the original problem formulation in terms
of bivariate products and one-dimensional functions by introduction of
auxiliary variables in place of intermediate nodes.
For example, consider $g(\mathbf{x'}) = \sin(x'_1x'_2)^2$.
By introduction of variables $\tilde{x}$ and $y$ we can rewrite the constraint
$g(\mathbf{x'}) \leq 0$ to $y^2 \leq 0$, $y = \sin(\tilde{x})$, and
$\tilde{x} = x'_1x'_2$.
Note that finite bounds on $\mathbf{x'}$ can be propagated to $\tilde{x}$
and~$y$.
For instance, $x'_1 \in [1, 2]$ and $x'_2 \in [0, \pi/4]$ result in
$\tilde{x} \in [0, \pi/2]$ and $y \in [0, 1]$.

Those reformulations, as performed by the modern \ac{minlp} solvers cited in the
introduction, allow us to focus on one-dimensional, general nonlinear
constraint functions only.
For this purpose, we denote them after reformulation as $f_j$,
\eg, $f_1(\tilde{x}) = \sin(\tilde{x})$,
and collect all linear and quadratic constraints (including bounds) as well as
integrality restrictions in a set $\Omega$.
Furthermore, vector $\mathbf{x}$ represents a concatenation of $\mathbf{x'}-$
and newly introduced $\mathbf{\tilde{x}}$-variables,
whereas the $\mathbf{y}$-variables are kept for clarity.
In conclusion, we assume our problem at hand to be presented as
\begin{equation}
	\begin{aligned}
		\min_{(\mathbf{x}, \mathbf{y}) \in \reals^n \times \reals^m} \ 
        & c(\mathbf{x}) \\
		\text{s.t.} \ & f_j(x_{i_j}) \leq y_j, & j \in [m], \\
		& (\mathbf{x}, \mathbf{y}) \in \Omega.
	\end{aligned}
	\tag{P}
	\label{eq:problem-exact}
\end{equation}
The number of $\mathbf{x}$-variables is now denoted by $n$,
whereas $m$ is the number of ``interesting'' constraints and
$\mathbf{y}$-variables.
The subindex $j$ of $i_j$ indicates that a variable $x_i$ may occur in
several constraints $j$.

Note that the constraint functions~$f_j$ are only defined on a one-dimensional
and finite domain, but remain general nonlinear ones.
Therefore, despite any reformulation, it cannot be expected to compute an exact
solution to~\labelcref{eq:problem-exact}.
Instead, in practice the user specifies an approximation tolerance $\eps > 0$
and is satisfied when obtaining an $\eps$-approximate solution, \ie,
a point that violates the nonlinear constraints by a value of at most $\eps$.
Such a point is equivalent to an exact optimal solution of the problem
\begin{equation}
	\begin{aligned}
		\min \ & c(\mathbf{x}) \\
		\text{s.t.} \ & f_j(x_{i_j}) \leq y_j + \eps, & j \in [m], \\
		& (\mathbf{x}, \mathbf{y}) \in \Omega.
	\end{aligned}
	\tag{P$_\eps$}
	\label{eq:problem-approx}
\end{equation}

For computing an $\eps$-approximate solution,
a natural idea is to substitute $f_j$ in~\labelcref{eq:problem-exact} by an
$\eps$-relaxation $\hat{f}_j$, \ie, it holds that $f_j - \hat{f}_j \leq \eps$
on the domain of $f_j$.
The resulting relaxed problem reads
\begin{equation}
    \begin{aligned}
        \min \ & c(\mathbf{x}) \\
        \text{s.t.} \ & \hat{f}_j(x_{i_j}) \leq y_j, & j \in [m], \\
        & (\mathbf{x}, \mathbf{y}) \in \Omega.
    \end{aligned}
    \tag{P$_\text{rel}$}
    \label{eq:problem-relax}
\end{equation}
A solution $(\mathbf{x}^\star, \mathbf{y}^\star)$
for~\labelcref{eq:problem-relax} then fulfills $f_j(x_{i_j}^\star) \leq
\hat{f}_j(x_{i_j}^\star) + \eps \leq y_j^\star + \eps$ for $j \in [m]$ and is
thus a valid $\eps$-approximate solution.

A strategy for computing a linear $\hat{f}_j$, which has reattained recent
attention, is the approach by piecewise linear functions.
In the following, we provide an introduction of it.

\subsection{Piecewise Linear Model}
\label{subsec:pwl-model}

We consider a nonlinear function $f$ and a finite
domain $\domain = [\lb x, \ub x]$ such that $\lb x < \ub x$.
First, we will explain the computation and modeling of a \acf{pwl}
\emph{approximation}, before extending it to a \ac{pwl} \emph{relaxation}.
A \ac{pwl} approximation $\bar{f}$ of $f$ can be constructed as a linear
interpolation of $f$ and is thus defined by breakpoints
$\{t_0, t_1, \dots, t_K\}$ with $t_0 = \lb x$, $t_K = \ub x$,
and $t_{k-1} < t_k$ for $k \in [K]$.
Parameter~$K \in \naturals$ denotes the number of intervals/linear pieces.
Then, for $x \in [t_{k-1}, t_k]$, it is
\begin{equation*} 
    \bar{f}(x) = f(t_{k-1}) \frac{t_k - x}{t_k - t_{k-1}} + 
                        f(t_k) \frac{x - t_{k-1}}{t_k - t_{k-1}}.
\end{equation*}

Given an approximation tolerance $\eps > 0$, we describe one way to
compute~$\bar{f}$.
Starting with $K = 1$, $t_0 = \lb x$, $t_1 = \ub x$, we check the condition
$\abs{\bar{f}(x) - f(x)} \leq \eps$ for all $x \in [t_0, t_1]$.
Since such a check can be performed by computing the maximum/minimum of the
difference, the aforementioned sufficient differentiability comes into play.
If the condition is met, $\bar{f}$ is a valid $\eps$-approximation of $f$.
Otherwise, we execute a binary search for the largest $t_1$ to fulfill the
condition.
After $t_1$ is determined, we increment~$K$ and start to find the
largest~$t_2$ such that the condition is met on $[t_1, t_2]$.
The procedure is iterated and stops when $t_K = \ub x$.
We refer to~\cite{braun2023pwl-t-rex} for a current implementation.

Besides its computation, the more important aspect is the incorporation of a
\ac{pwl} approximation $\bar{f}$ into an optimization model.
To the best of our knowledge, a \ac{pwl} formulation was first
introduced in~\cite{markowitz1957orig-intro-pwl}.
The method explained there is referred to as \emph{incremental method}
or \emph{$\delta$-method}.
In the meantime, there have been established several adjusted
formulations, see~\cite{beach2024nonconvex-miqcp-to-mip,
    burlacu2020mip-for-minlp,
    croxton2003pwl-comparison,
    huchette2023pwl-new-formulations,
    rebennack2015pwl-2d},
that try to improve computational performance, \eg, by reducing the number of
variables in the model.
Recently though, \citeauthor{braun2023pwl-t-rex}~\cite{braun2023pwl-t-rex}
demonstrated that in some settings the incremental method is still superior to
all other methods in terms of computational performance.
Therefore, we choose it for our computational comparison later on and explain it
in the following, rephrasing mainly from~\cite{braun2023pwl-t-rex}.

For a given $x \in \domain$, we first need to determine the corresponding
interval index $\bar{k} \in [K]$ such that $x \in [t_{\bar{k}-1}, t_{\bar{k}}]$.
In the incremental method, we introduce one binary variable~$u_k$ 
for each but the last interval $[t_{k-1}, t_k]$, $k \in [K - 1]$, 
and aim to have $u_k = 1$ for all intervals with $t_k \leq x$, else $u_k = 0$.
In other terms, the binary variable~$u_k$ is activated whenever $x$ exceeds
the upper bound of the corresponding $k$th interval until the interval
containing~$x$ is reached.
Formally, this can be ensured by
\begin{align*}
    & t_0 + \sum_{k = 1}^{K - 1} u_{k} (t_k - t_{k -1}) \leq x,\\
    & u_{k+1} \leq u_k, \qquad \text{for } k \in [K-1].
\end{align*}
Besides determining the relevant interval $[t_{\bar{k} - 1}, t_{\bar{k}}]$,
we second need to compute the relative distance of $x$ inside it.
Here, the incremental method advocates to introduce 
continuous variables $\delta_k \in [0, 1]$ corresponding to the intervals,
$k \in [K]$.
As for the binary variables $u_k$, the continuous ones $\delta_k$ are also
set to one when $x$ exceeds~$t_k$, which is
enforced by the inequalities
\begin{equation}
    \label{eq:pwl-increment}
    \delta_{k+1} \leq u_k, \qquad \text{for } k \in [K-1].
\end{equation}
The continuous variables even allow to model the value of $x$ explicitly as
\begin{equation}
    \label{eq:pwl-x-value}
    x = t_0 + \sum_{k = 1}^{K} \delta_k (t_k - t_{k-1}),
\end{equation}
whereas for the interpolation value $w$ of $f$, we get
\begin{equation}
    \label{eq:pwl-w-value}
    w = f(t_0) + \sum_{k = 1}^K \delta_k(f(t_k) - f(t_{k-1})).
\end{equation}
In conclusion, the incremental method poses a model for \ac{pwl}
approximation of the form
\begin{equation}
\begin{aligned}
    \labelcref{eq:pwl-increment}, &\labelcref{eq:pwl-x-value}, 
    \labelcref{eq:pwl-w-value}, \\
    \delta_1 &\leq 1,\\
    u_k &\leq \delta_k, & k \in [K-1],\\
    \delta_K &\geq 0,\\
    u_k &\in \{0, 1\}, & k \in [K-1],
\end{aligned}
\label{eq:pwl-approx}
\end{equation}
where the edge cases are respected as well.
Note that the binary variables are necessary to fix the $\delta$-variables
to their respective bounds such that only the relevant one is potentially
fractional.

Since this main ingredient of ``filling'' or ``incrementation'' is often modeled
via variables named $\delta$, the method is also known by this respective name.

Note that system~\labelcref{eq:pwl-approx} models a \ac{pwl} 
\emph{approximation} of $f$ such that $\abs{w - f} \leq \eps$
on~$\domain$.
Thus, a simply replacement of each $f_j$ in~\labelcref{eq:problem-exact}
with a corresponding $w_j$ and~\labelcref{eq:pwl-approx} does in general not 
pose a relaxation~\labelcref{eq:problem-relax},
but an approximated problem.
Consequently, if the approximated problem is solved, its solution
is not guaranteed to be also a solution of~\labelcref{eq:problem-approx}.
However, to retrieve a relaxation, we can define $\eps' \define \eps/2$,
compute a \ac{pwl} approximation with respect to $\eps'$,
and shift the result for $\eps/2$.
In particular, the \ac{pwl} approximation then ensures that 
$\abs{w - f(x)} \leq \eps'$ for all $x \in \domain$.
This implies that $w - \eps' \leq f(x)$ and thus $\abs{(w - \eps') - f(x)} \leq
\abs{w - f(x)} + \eps' \leq 2\eps' = \eps$.
In other terms, a shift of the function values $f(t_k)$
in~\labelcref{eq:pwl-w-value} by $\eps'$ is a shift in $w$ and
results in a underestimation -- thus relaxation
like~\labelcref{eq:problem-relax} -- satisfying a tolerance of~$\eps$.
Such a procedure for obtaining a \ac{pwl} relaxation is performed throughout
the present article and we refer to
this relaxation technique simply by \ac{pwl}.

\subsection{Global One-sided Parabolic Model}
\label{subsec:para-model}

The recently introduced approach of \acf{paralong} approximations and
relaxations~\cite{goess2025para} aims to circumvent the use of binary variables
as in \ac{pwl}.
For a given constraint function $f$, it computes an approximation $\breve{f}$
as the point-wise maximum of paraboloids, \ie, separable univariate quadratic
functions, that underestimate $f$ on~$\domain$ individually and satisfy the
given approximation tolerance.
As we consider one dimensional constraint functions~$f$ only, we refer to this
type of approximates as parabolas.
Formally, for a set of $K \in \naturals$ parabolas $\{p_k\}_{k\in [K]}$,
the approximation is $\breve{f}(x) = \max_{k \in [K]} p_k(x)$ such that
\begin{equation}
    \label{eq:main-properties}
    f - \eps \leq \breve{f} \leq f \qquad \text{ on } \domain.
\end{equation}
In other terms, $\breve{f}$ represents an epigraph approximation of $f$.
Due to the global underestimation, we can simply replace each constraint
in~\labelcref{eq:problem-exact} by a corresponding parabolic approximation
and end up with~\labelcref{eq:problem-relax}.
Taking the presentation as `$\leq$'-constraints and the pointwise maximum into
account, it even simplifies to the set of constraints
$p_k(x) \leq y$, $k \in [K]$.
This makes the modeling of \ac{para} relaxations straightforward and does
not rely on the use of binary variables.

It remains to clarify the computation of $p_k$ in~\cite{goess2025para}.
For a constraint function~$f$ and its domain $\domain = [\lb x, \ub x]$, the
approach requires a global Lipschitz constant $L > 0$ of $f$.
This can be considered a weaker assumption in comparison to continuous
differentiability, since maximizing $\abs{f'}$ on $\domain$ gives~$L$.
Based on the magnitude of~$L$, the approach in~\cite{goess2025para}
performs a binary search on the number of parabolas $K$.
During one search iteration, $K$ is fixed and a \ac{mip} problem is constructed
in which the coefficients of the parabolas are considered as variables.
The problem's constraints ensure the $\eps$-approximation property of a
solution, whereas the objective controls the global underestimation, \ie,
the first and second inequality in~\labelcref{eq:main-properties}, respectively.
If the \ac{mip} problem is solved in a given time limit with objective of zero,
the corresponding parabolas fulfill~\labelcref{eq:main-properties} and the
current $K$ is decreased.
Otherwise, the number of parabolas $K$ is increased.

The weak assumptions necessary for the approach represent a clear advantage,
as it leaves it applicable for a really broad class of constraint functions.
However, relying on the solution of several \ac{mip} problems can be
computationally challenging and results in function-domain-$\eps$ combinations
that are not computable, see the comparison below
in~\Cref{subsec:comp-results-para} for more details.

In fact, the article~\cite{goess2025para} introduces a theoretically valid
as well as a more practical option of this approach.
The latter circumvents some of the challenges by relaxing particular constraints
of the \ac{mip} and by incorporating manual checks.
Those changes turn it into a computationally more efficient version.
Hence, we conduct the later comparison in~\Cref{subsec:comp-results-para}
with respect to this version and refer to it as \texttt{practical-MIP},
following the original source.

\section{Computing Parabolas Fast}
\label{sec:algo}

Inspired by the computation of \ac{pwl} approximations, we introduce a novel
method to compute \ac{para} approximations for general nonlinear functions $f$
in an iterative manner.
We recall that $f$ is assumed to be sufficiently continuous differentiable and
defined on a domain $\domain = [\lb x, \ub x]$.
An explanation of the method and respective proofs of correctness are provided
in~\Cref{subsec:method}.
We refer this procedure as the \ac{para} approximation method or
shortly~\ac{para}.
It  is indeed able to compute parabolas for cases where the original
method from~\cite{goess2025para} has failed,
see~\Cref{subsec:comp-results-para}.
The computational performance even allows to conclude the present section
with a comparison of number of approximates, \ie, parabolas and linear pieces,
on general nonlinear functions that occur frequently in the
MINLPLib~\cite{bussieck2003minlplib}.

\subsection{Method}
\label{subsec:method}

In contrast to the method in~\cite{goess2025para}, we associate one
parabola $p_k$, ${k \in [K]}$, with a certain domain 
${\domain_k \define [t_{k-1}, t_k]}$, on which it fulfills the
$\eps$-approximation, \ie,
\begin{equation}
	\label{eq:eps-approx}
	f(x) - \eps \leq p_k(x), \qquad \text{for all } x \in \domain_k.
\end{equation}
To construct a global $\eps$-approximation with all parabolas,
they must satisfy the first inequality in~\labelcref{eq:main-properties}
and we thus require
\begin{equation}
    \label{eq:domain-coverage}
	\bigcup_{k \in [K]} \domain_k = \domain.
\end{equation}
Simultaneously, as described above, each $p_k$ is supposed to serve as a
global underestimator of $f$, \ie,
\begin{equation}
	\label{eq:global-under}
	p_k(x) \leq f(x), \qquad \text{for all } x \in \domain.
\end{equation}
This settles the description of the requirements.

Our new approach returns the parabolas ``from left to right'',
\ie, it computes a parabola which satisfies property~\labelcref{eq:eps-approx}
on $[\lb x, t_1] = [t_0, t_1]$, and then proceeds to $[t_1, t_2]$ until it reaches the
upper bound $\ub x$ with the final interval $[t_{K-1}, t_K] = [t_{K-1}, \ub x]$. 
Here, $K \geq 0$ indicates the necessary number of parabolas, which is
determined a-posteriori.
Such a procedure can be leveraged for computing \ac{pwl} approximations
as well, see~\Cref{subsec:pwl-model}.

Let's turn to a particular iteration $k$ and investigate the
challenge at hand.
The current lower bound is fixed at $t_{k-1} < \ub x$.
To cover a large interval, on which $p_k$ is a valid $\eps$-approximation,
we aim to maximize the next upper bound~$t_k \leq \ub x$.
Enforcing~\labelcref{eq:eps-approx} and~\labelcref{eq:global-under},
we can formulate this as the semi-infinite optimization problem
\begin{subequations}
\begin{align}
		\max_{t_k,\, \tilde{a},\, \tilde{b},\, \tilde{c}} \ & t_k \notag\\
		\text{s.t.} \ & \tilde{a}x^2 + \tilde{b}x + \tilde{c} \geq f(x) - \eps,
        && \text{for all } x \in [t_{k-1}, t_k] = \domain_k,
        \label{eq:k-th-problem-eps} \\
        & \tilde{a}x^2 + \tilde{b}x + \tilde{c} \leq f(x), && 
        \text{for all } x \in \domain,\label{eq:k-th-problem-under}
\end{align}
\label{eq:k-th-problem}
\end{subequations}
where we define $t_{0} = \lb x$ for $\domain_1$ and
$(\tilde{a}, \tilde{b}, \tilde{c})$ denote variables for the coefficients
of~$p_k$.
A solution $(a_k, b_k, c_k)$ to this problem then gives 
$p_k(x) \define a_kx^2 + b_kx + c_k$, satisfying~\labelcref{eq:eps-approx}
and~\labelcref{eq:global-under} by~\labelcref{eq:k-th-problem-eps}
and~\labelcref{eq:k-th-problem-under}, respectively.
However, the constraints~\labelcref{eq:k-th-problem}
are infinitely many and thus require further treatment.

In a first attempt let us assume to have a fixed parabola $p_k$
and a fixed upper bound $t_k$.
Then, computing 
\begin{equation}
	\label{eq:sub-global-function-bound}
	\max_{x \in \domain} p_k(x) - f(x) = 
    \max_{x \in \domain} a_kx^2 + b_kx + c_k - f(x),
\end{equation}
and
\begin{equation}
	\label{eq:sub-local-approx-bound}
	\max_{x \in \domain_k} f(x) - p_k(x) = 
    \max_{x \in \domain_k} f(x) - (a_kx^2 + b_kx + c_k ),
\end{equation}
can be considered computational tractable, since both problems require to
maximize differences of continuously differentiable functions on a
closed interval.
This can be conducted by means of Newton-type methods, for instance.

Let $z_k$ and $z^\eps_k$ denote the solutions
for~\labelcref{eq:sub-global-function-bound} and
\labelcref{eq:sub-local-approx-bound}, respectively.
Then, \labelcref{eq:k-th-problem-eps} is fulfilled if and only if
$f(z_k^\eps) - p_k(z_k^\eps) \leq \eps$, and, analogously,
\labelcref{eq:k-th-problem-under} is fulfilled if and only if
$p_k(z_k) - f(z_k) \leq 0$.
In the case that at least one of these inequalities does not hold, a decrease of
the interval length $\abs{\domain_k}$, thus a shift of $t_k$ towards $t_{k-1}$,
and a new computation of $p_k$ may resolve the issue.

Based on this observation, we propose a two-step or \emph{inner-outer}
approach.
Let $k$ be the current iteration index, then
\begin{itemize}
	\item an outer loop tries to compute a maximal upper bound $t_k$ for
	the current interval $[t_{k-1}, t_k]$, but iteratively reduces the value 
	of $t_k$ if the inner loop fails, shifting $t_k$ in the direction
    of~$t_{k-1}$, and
	\item an inner loop aims to find a feasible set of coefficients
    $(a_k, b_k, c_k)$, defining the current parabola $p_k$ such 
    that~\labelcref{eq:k-th-problem} hold true.
\end{itemize}

\subsubsection{Inner Loop}
\label{subsec:inner-loop}

We start our explanation with the inner loop.
For ease of notation, we drop the iteration index $k$ and thus write the current
(local) domain $\domain_k = \locdomain = [\lb t, \ub t]$ as well as
$p_k(x) = p(x) = ax^2 + bx + c$.

Throughout the inner loop, we force $p$ to cross $f - \eps$ at the bounds
of $\locdomain$.
In particular, this can be expressed as
	\begin{align*}
		p(\lb t) = f(\lb t) - \eps,\\
		p(\ub t) = f(\ub t) - \eps,
	\end{align*}
which reads in explicit form as
\begin{subequations}
	\label{eq:b-c-computation-explicit}
	\begin{align}
		a{\lb t}^2 + b\lb t + c = f(\lb t) - \eps, 
        \label{eq:b-c-computation-explicit-1}\\
		a{\ub t}^2 + b\ub t + c = f(\ub t) - \eps.
        \label{eq:b-c-computation-explicit-2}
	\end{align}
\end{subequations}
Neglecting~\labelcref{eq:k-th-problem-under} for the moment,
the intuition behind this measure is as follows:
If the parabola $p$ fulfills~\labelcref{eq:k-th-problem-eps} and
simultaneously $p(\lb t) > f(\lb t) - \eps$ or $p(\ub t) > f(\ub t) - \eps$,
the domain $\locdomain$ could have been chosen larger.

Fixing $p$ at the bounds of $\locdomain$ reduces to original degrees of
freedom when determining the three coefficients of $p$ by two.
In other words, given one of the coefficients $a$, $b$, or $c$, 
this system uniquely determines the remaining two. 
To see this, assume coefficient $a$ as fixed and rewrite the
system~\labelcref{eq:b-c-computation-explicit} to $A \vv = \vd$, where
\begin{equation*}
	A = \begin{pmatrix}
		\lb t & 1 \\
		\ub t & 1
	\end{pmatrix},
	\qquad	
	\vd = \begin{pmatrix} 
		f(\lb t) - \eps - a{\lb t}^2 \\
		f(\ub t) - \eps - a{\ub t}^2
	\end{pmatrix},
\end{equation*}
and $\vv = \begin{pmatrix} b & c \end{pmatrix}^\top$.
When $\locdomain \neq \emptyset$, so $\lb t < \ub t$, 
it holds $\det(A) = \lb t - \ub t < 0$ and the system always 
attains a unique solution.

Instead of fixing coefficient $a$ to determine $(b, c)$, we 
combine~\labelcref{eq:b-c-computation-explicit} 
with~\labelcref{eq:k-th-problem} and can derive
bounds to the quadratic coefficient $a$.
This will allow to determine the existence of a feasible set of coefficients
$(a, b, c)$ and constitute the main ingredient of the inner loop.
In the following, the derivation of bounds is formalized.

\begin{lemma}[Bounds on $a$]
	\label{lem:bounds-a}
	Let $f : \domain \to \R$ be  as above and
    $\mathrm{int}(\domain) \neq \emptyset$.
	Further, let $\eps > 0$ and 
	$\locdomain \define [\lb t, \ub t] \subset \domain$ such that
    $\lb x < \lb t < \ub t < \ub x$, 
	implying $\mathrm{int}(\locdomain) \neq \emptyset$.
	For a parabola $p(x) = ax^2 + bx + c$, 
    assume~\labelcref{eq:b-c-computation-explicit}.
	Then the following implications hold true.
	\begin{itemize}
		\item[(a)] Consider $\widehat{x} \in \domain\setminus\locdomain$.
		Then $p(\widehat x) \leq f(\widehat x)$ if and only if
		\begin{equation*}
			a \leq \frac{f(\widehat x) - f(\lb t) + \eps}
            {(\widehat x - \lb t)(\widehat x - \ub t)} - 
            \frac{f(\ub t) - f(\lb t)}{(\ub t - \lb t) ( \widehat x - \ub t)}
            \enifed A(\widehat{x}).
		\end{equation*}
		\item[(b)] Consider $\widehat{x} \in \mathrm{int}(\locdomain)$. 
		Then $p(\widehat x) \leq f(\widehat x)$ if and only if
		\begin{equation*}
			a \geq \frac{f(\widehat x) - f(\lb t) + \eps}
            {(\widehat x - \lb t)(\widehat x - \ub t)} - 
            \frac{f(\ub t) - f(\lb t)}{(\ub t - \lb t) ( \widehat x - \ub t)}
            \enifed A(\widehat{x}).
		\end{equation*}
		\item[(c)] Consider $\widehat{x} \in \mathrm{int}(\locdomain)$. 
		Then $p(\widehat x) \geq f(\widehat x) - \eps$ if and only if
		\begin{equation*}
			a \leq \frac{f(\widehat x) - f(\lb t)}
            {(\widehat x - \lb t)(\widehat x - \ub t)} - 
            \frac{f(\ub t) - f(\lb t)}{(\ub t - \lb t) ( \widehat x - \ub t)}
            \enifed B(\widehat{x}).
		\end{equation*}
	\end{itemize}
\end{lemma}

\begin{proof}
	As a preparatory step, we subtract~\labelcref{eq:b-c-computation-explicit-1}
	from~\labelcref{eq:b-c-computation-explicit-2} and get
	\begin{equation*}
		a(\ub{t}^2 - \lb{t}^2) + b(\ub{t} - \lb{t}) = f(\ub{t}) - f(\lb{t}),
	\end{equation*}
	which is equivalent to 
	\begin{equation}
		a(\ub{t} + \lb{t}) + b = \frac{f(\ub{t}) - f(\lb{t})}{\ub{t} - \lb t}
        \label{eq:cond-diff}
	\end{equation}
	by division with $(\ub{t} - \lb{t})$.
	This preliminary work will support the proof for every of the three cases.
	
	\paragraph{Case (a)}
	Starting off, we write the condition from (a) explicitly, \ie,
	\begin{equation*}
		a\widehat{x}^2 + b\widehat{x} + c \leq f(\widehat{x}).
	\end{equation*}
	By subtracting~\labelcref{eq:b-c-computation-explicit-1}, this is equivalent
    to
	\begin{equation*}
		a(\widehat{x}^2 - \lb{t}^2) + b(\widehat{x} - \lb{t}) 
        \leq f(\widehat{x}) - f(\lb{t}) + \eps.
	\end{equation*}
	Now, rearranging~\labelcref{eq:cond-diff} for $b$ and 
	plugging the result into the former gives
	\begin{align*}
		& a(\widehat{x}^2 - \lb{t}^2) + 
        \left(\frac{f(\ub{t}) - f(\lb{t})}{\ub{t} - \lb t} - 
        a(\ub{t} + \lb{t}) \right) (\widehat{x} - \lb{t}) 
        \leq f(\widehat{x}) - f(\lb{t}) + \eps \\
		\iff & a\left(\widehat{x}^2 - \lb{t}^2 - \ub{t}\widehat{x} - 
        \lb{t}\widehat{x} + \ub{t}\lb{t} + \lb{t}^2\right)  
        \leq f(\widehat{x}) - f(\lb{t}) + \eps - 
        \frac{f(\ub{t}) - f(\lb{t})}{\ub{t} - \lb t}(\widehat{x} - \lb{t}) \\
		\iff & a(\widehat{x} - \ub{t})(\widehat{x} - \lb{t}) 
        \leq f(\widehat{x}) - f(\lb{t}) + \eps - 
        \frac{f(\ub{t}) - f(\lb{t})}{\ub{t} - \lb t}(\widehat{x} - \lb{t}).
	\end{align*}
	As $\widehat{x} \in \domain\setminus\locdomain$, either 
    $\widehat{x} < \lb t < \ub t$ or $\lb t < \ub t < \widehat{x}$.
	That is, $(\widehat{x} - \ub{t})(\widehat{x} - \lb{t}) > 0$ and thus 
    division by this term does not change the direction of the inequality
    but leads to~(a).
	
	\paragraph{Case (b)}
	This is analogous to case (a), however, 
	the last division changes the sign of the inequality, 
	as $\widehat{x} \in \mathrm{int}(\locdomain)$ and 
    thus $(\widehat{x} - \ub{t})(\widehat{x} - \lb{t}) < 0$.
	
	\paragraph{Case (c)}
	Again, we write the condition from (c) explicitly and get
	\begin{equation*}
		a\widehat{x}^2 + b\widehat{x} + c \geq f(\widehat{x}) - \eps.
	\end{equation*}
	Subtraction of~\labelcref{eq:b-c-computation-explicit-1} gives
	\begin{equation*}
		a (\widehat{x}^2 - \lb{t}^2) + b(\widehat{x} - \lb{t})
        \geq f(\widehat{x})  - f(\lb{t}).
	\end{equation*}
	Following the strategy from previous cases and respecting that
    $\widehat{x} \in \mathrm{int}(\locdomain)$ gives the desired inequality.
\end{proof}

\Cref{lem:bounds-a} implies that every point in $\domain\setminus\locdomain$
implicitly gives an upper bound to the quadratic coefficient $a$ and every point
inside $\locdomain$ even implies lower and upper bounds to $a$.
Setting $\lb a \gets \sup_{x \in \mathrm{int}(\locdomain)} A(x)$ and
${\ub a \gets \min\left(\inf_{x \in \domain\setminus\locdomain} A(x),
\inf_{x \in \locdomain}B(x)\right)}$, we ensure the existence of a
feasible coefficient $a$ (thus parabola) if and only if $\lb a \leq \ub a$.
A feasible $a$ can then be completed computing $(b, c)$ by
solving the linear system~\labelcref{eq:b-c-computation-explicit}.
In turn, if we know that $\lb a > \ub a$, there exists no feasible $a$ and
we need to shrink~$\locdomain$ in the outer loop.

The computation of the minima and maxima above to determine $\lb{a}$ and
$\ub{a}$ is non-trivial, since the variable $x$ appears in the denominator of a
fraction, especially in a term that converges to infinity.
Hence, we compute these bounds iteratively, starting with some initial
values~$\lb{a}_0$ and~$\ub{a}_0$.
In fact, as any $a \in [\lb a, \ub a]$ is valid, we initialize $\lb{a}_0$ with
negative infinity and always set $a$ to be the current upper bound.
This will facilitate the presentation of the method and showing its correctness.
However, it remains to find a value for $\ub{a}_0$, which can also be
derived by the results of~\Cref{lem:bounds-a} as mentioned below.

\begin{remark}
	\label{rem:initial-a-conditions}
	Consider a point $x \in \mathrm{int}(\locdomain)$ and
    \Cref{lem:bounds-a}(c).
	If $x \to \lb t$, the inequality bounding coefficient $a$ turns into
	\begin{equation*}
		a \leq \frac{f'(\lb t)}{\lb t - \ub t} + 
        \frac{f(\ub t) - f(\lb t)}{(\ub t - \lb t)^2} 
        = (f(\ub t) - f(\lb t) + (\lb t - \ub t)f'(\lb t))/(\ub t - \lb t)^2.
	\end{equation*}
    This directly gives an upper bound to $a$, which can be used for $\ub{a}_0$.
    
    For the case $x \to \ub t$, we first need to rewrite $B(x)$.
    By respecting that
    \begin{equation*}
       D \define 
       \frac{1}{(x - \lb t)(x - \ub t)} - \frac{1}{(\ub t - \lb t)(x - \ub t)}
        = - \frac{1}{(\ub t - \lb t)(x - \lb t)},
    \end{equation*}
    we can derive that
    \begin{align*}
        B(x) &
        = \frac{f(x) - f(\lb t)-f(\ub t) + f(\ub t)} {(x - \lb t)(x - \ub t)} -
        \frac{f(\ub t) - f(\lb t)}{(\ub t - \lb t) (x - \ub t)} \\
        & =  \frac{f(x) - f(\ub t)}{(x - \lb t)(x - \ub t)}
        + (f(\ub t) - f(\lb t)) D\\
        & = \frac{f(x) - f(\ub t)}{(x - \lb t)(x - \ub t)} -
        \frac{f(\ub t) - f(\lb t)}{(\ub y - \lb t)(x - \lb t)}.
    \end{align*}
    Then, considering $x \to \ub t$ for $x \in \mathrm{int}(\locdomain)$,
    \Cref{lem:bounds-a}(c) provides that
	\begin{equation*}
		a \leq (f(\lb t) - f(\ub t)+(\ub t - \lb t)f'(\ub t))/(\ub t - \lb t)^2.
	\end{equation*}
    Taking the minimum of both those upper bounds, leads to $\ub{a}_0$
    for our inner loop.
\end{remark}

After initializing the quadratic coefficient $a$, the remaining
coefficients~$(b, c)$ can be computed by solving
\labelcref{eq:b-c-computation-explicit}.
Then, the inner loop tests for~\labelcref{eq:eps-approx}
and~\labelcref{eq:global-under} by 
computing~\labelcref{eq:sub-global-function-bound}
on $\domain\setminus\locdomain$ and $\locdomain$ separately,
as well as~\labelcref{eq:sub-local-approx-bound} on $\locdomain$.
Explicitly, for a fixed parabola $p$, this is
\begin{subequations}
	\label{eq:inner-loop-maxima}
	\begin{align}
		\max_{x \in \domain\setminus\locdomain} p(x) - f(x), 
        \label{eq:inner-loop-maximum-outside}\\
		\max_{x \in \mathrm{int}(\locdomain)} p(x) - f(x),
        \label{eq:inner-loop-maximum-in-under}\\
		\max_{x \in \mathrm{int}(\locdomain)} f(x) - p(x) - \eps.
        \label{eq:inner-loop-maximum-in-eps}
	\end{align}
\end{subequations}
We denote the solutions by $x^1,\, x^2,\, x^3$ and their values by
$v_1,\, v_2, \, v_3$, respectively.
If $v_{\max} \define \max\{v_1, v_2, v_3\}$ exceeds the threshold of zero,
the current choice of coefficient $a$ is infeasible and
$x^1,\, x^2,\,x^3$  are considered as candidates for ``$\widehat{x}$''
in~\Cref{lem:bounds-a}, implying an update of the current
bounds~$\lb{a}_l$ and $\ub{a}_l$ on $a$.

The process is iterated until the bounds are contradictory, 
signaling the absence of a feasible triple $(a, b, c)$, or until the 
constraints~\labelcref{eq:k-th-problem} are satisfied.
This process is summarized in Method~\labelcref{alg:subroutine}.
Arrows are always pointing towards the object that is returned by this step.

\begin{algorithm}[h]
	\caption{Inner loop -- Computing $(a, b, c)$.}\label{alg:subroutine}
	\begin{algorithmic}[1]
		\Require (Local) Interval bounds $\lb t$ and $\ub t$ such that 
        $\lb t < \ub t$.
		\Ensure Coefficients $(a, b, c)$ defining $p$
        satisfying~\labelcref{eq:k-th-problem} or ``Infeasible''.
  		\State Initialize $\ub{a}_0$ according to
          \Cref{rem:initial-a-conditions}, $\lb{a}_0 \gets -\infty$,
          and $l \gets 0$.
		\State Set $a_0 \gets \ub{a}_l$ and compute $(b_l, c_l)$ according
        to~\labelcref{eq:b-c-computation-explicit}.
		\While{$\lb{a}_l \leq \ub{a}_l$}
			\State Compute~\labelcref{eq:inner-loop-maxima} 
            $\to x^1,\, x^2,\, x^3,\, v_{\max}$.
			\If{$v_{\max} \leq 0$}
				\State \Return $(a_l, b_l, c_l)$.
                \label{alg:subroutine:step:return}
			\EndIf
			\State Update $\lb{a}_{l+1}$ and $\ub{a}_{l+1}$ with $x^1,\, x^2,\, x^3$
            according to~\Cref{lem:bounds-a}.
            \State Increment $l \gets l+1$.
            \State Set $a_l \gets \ub{a}_l$ and compute $(b_l, c_l)$ according
            to~\labelcref{eq:b-c-computation-explicit}.
		\EndWhile
		\State \Return ``Infeasible''.
	\end{algorithmic}
\end{algorithm}

After its presentation, we continue by formally showing
the correct mode of operation of the inner loop.

\begin{restatable}[Correctness of Method~\labelcref{alg:subroutine}]{theorem}
    {thmone}
    \label{thm:correctness-inner-loop}
	If there exists a solution $(a, b, c)$
	fulfilling~\labelcref{eq:k-th-problem}
	and~\labelcref{eq:b-c-computation-explicit},
    Method~\labelcref{alg:subroutine} returns one or converges to it.
	Otherwise, it returns that there is no such solution.
\end{restatable}

The following provides a proof sketch, for which the details can be found in
\Cref{subsec:proof-inner-loop}.

\begin{proof}
    By leveraging \Cref{lem:bounds-a} the existence of finite bounds
    $\lb a \leq \ub a$ is equivalent to the existence of a solution $(a, b, c)$.
    This represents the main ingredient.
    
    If there exists a solution, the condition of the while-loop is always met.
    Since $a_l$ in iteration $l$ is chosen to be the current $\ub{a}_l$, 
    also $a_l \geq \lb a$ for all $l$.
    By \Cref{lem:bounds-a}, it thus always holds that $p \leq f$
    on~$\mathrm{int}(\locdomain)$.
    The remaining cases, $p \leq f$ on~$\domain\setminus\locdomain$ and
    $p \geq f - \eps$ on~$\mathrm{int}(\locdomain)$, are proven following the
    same strategy:
    We use~\labelcref{eq:b-c-computation-explicit} to write down
    a parameterized version of parabola $p$ in the quadratic coefficient $a$.
    This allows us to show convergence by contradiction.
\end{proof}

\subsubsection{Outer Loop}
\label{subsec:outer-loop}

While the inner loop computes an approximating parabola for a given interval
or asserts with its impossibility, it remains to consider the interval bounds.
We propose to start with the global interval $[t_0, t_1] \define [\lb x, \ub x]$
and consecutively shift $t_1$ toward $t_0$, until a suitable parabola
can be computed by the inner loop.
After storing the parabola and the associated domain $[t_0, t_1]$,
the remaining interval $[t_1, \ub x]$ is considered next.
This is iterated until $K$ parabolas~$p_k$ with their associated 
domains~$\domain_k$ are computed, covering the entire global domain~$\domain$,
\ie, fulfilling~\labelcref{eq:domain-coverage}.
This procedure is presented in pseudocode in Method~\labelcref{alg:main}.

\begin{algorithm}[h]
	\caption{Outer loop -- Computing $\domain_k$}\label{alg:main}
	\begin{algorithmic}[1]
		\Require (Global) Domain $\domain = [\lb x, \ub x] \neq \emptyset$,
        approximation tolerance $\eps > 0$.
        \Ensure A set of coefficients $(a_k, b_k, c_k)$ defining $p_k$,
        $k \in [K]$, $K \in \naturals$,
        such that~\labelcref{eq:main-properties}.
		\State Initialize $k \gets 1$, $t_0 \gets \lb x$, $t_1 \gets \ub x$.
		\While{$t_{k-1} < \ub x$}
			\State Call Method~\labelcref{alg:subroutine} on $[t_{k-1},t_k]$
            $\to (a, b, c)$ or ``Infeasible''.
			\If{``Infeasible''}
				\State Shift $t_k$ towards $t_{k-1}$.\label{alg:main:step:shift}
			\Else
				\State Store $(a_k, b_k, c_k)$.
				\State Set $t_{k+1} \gets \ub x$ and increment $k \gets k + 1$.
			\EndIf
		\EndWhile
		\State \Return Parameters $(a_k, b_k, c_k)$ for $k \in [k]$.
	\end{algorithmic}
\end{algorithm}

Depending on the type of shift applied in Step~\labelcref{alg:main:step:shift},
the number of parabolas and thus iterations varies.
We want to briefly discuss this effect by supposing a shift
in form of a convex combination, \ie,
$t_k \gets \lambda t_{k-1} + (1 - \lambda) t_k$ for $\lambda \in (0, 1)$.
If $\lambda$ is near 1, the shift is strong towards $t_{k-1}$ and 
the length of the current interval~$\domain_k$ is decreased rapidly.
In consequence, it is likely to find a suitable parabola
with~\labelcref{eq:eps-approx} and~\labelcref{eq:global-under}
after only a couple of ``Infeasible''-iterations of the inner loop.
However, it takes a considerable number $K$ of intervals -- thus parabolas --
to cover the entire domain, \ie, to reach~\labelcref{eq:domain-coverage}.
In turn, a value of $\lambda$ close to 0 results in the opposite effect,
potentially computing a small number of parabolas -- thus few quadratic
constraints for the \ac{para} relaxation -- but for a certain
computational cost.

The specific number of calls of the inner loop depends on the function~$f$
and its domain~$\domain$ to be considered, as well as on the value of $\lambda$.
From the theoretical side, however, it remains to show that
Method~\labelcref{alg:main} converges \emph{at all}.
In the following, we will demonstrate it,
supposing the shift in Step~\labelcref{alg:main:step:shift} to be a true
decrease of the interval length of $\domain_k$ that still ensures
$\abs{\domain_k} > 0$.
In particular, we show that for small enough domains~$\domain_k$,
there exists a respective solution $p_k$.
This allows to derive by \Cref{thm:correctness-inner-loop} that the inner loop
at least converges to such a solution.
Considering numerical precision and thus finiteness of the inner loop,
we can conclude that Method~\labelcref{alg:main} terminates.

Note that we will only leverage Lipschitz continuity of our function $f$ for
this proof.
As noted in~\Cref{subsec:para-model}, a global Lipschitz constant $L$ can be
computed as (an upper bound to) $\max\{ \abs{f'(x)} \mid x \in \domain\}$,
\ie, the assumptions of the theorem apply to the given setup.
Keeping the assumptions on~$f$ low, however, allows for wider application
and a comparison of the proof techniques to the ones in~\cite{goess2025para}.

\begin{restatable}[Correctness of Method~\labelcref{alg:main}]{theorem}{thmtwo}
    \label{thm:correctness-outer-loop}
	Let $\domain \neq \emptyset$, $f: \domain \to \reals$ be globally Lipschitz
    continuous with $L > 0$ and $\eps > 0$.
    Then, for any interval length $\Delta > 0$ with
    $\Delta \leq \eps/(3L)$ and $\abs{\domain}/\Delta \in \naturals$,
    one can choose $\domain_k$ as a partition of $\domain$ with
    $\abs{\domain_k} = \Delta$ and by setting $a_k = -4L/\Delta$
    there exist parabolas
    $p_k(x) = a_kx^2 + b_k + c_k$, $k \in [\abs{\domain}/\Delta]$,
    fulfilling the properties~\labelcref{eq:eps-approx} 
    and~\labelcref{eq:global-under}.
\end{restatable}

The following provides a proof sketch, since the full proof is of technical
nature.
Details can be found in \Cref{subsec:proof-outer-loop}.

\begin{proof}
    For a given local interval $\locdomain$ with length $\Delta$, we show
    the existence of a parabola $p$ such that
    \begin{itemize}
        \item[a)] $p(x) \geq f(x) - \eps$ on $\locdomain$,
        \item[b)] $p(x) \leq f(x)$ on $\locdomain$, and
        \item[c)] $p(x) \leq f(x)$ on $\domain\setminus\locdomain$.
    \end{itemize}
    For case a), we make use of~\Cref{lem:bounds-a} and show that
    $\tilde{a} \leq B(x)$ for all $x \in \mathrm{int}(\locdomain)$.
    Cases b) and c) are straightforward, using the Lipschitz continuity
    and the assumptions to~$\Delta$.
    Denoting the mid point of $\locdomain$ as $\tilde{t}$,
    all three cases distinguish $x \leq \tilde{t}$ and $x \geq \tilde{t}$.

\end{proof}

\subsection{Computational Results}
\label{subsec:comp-results-para}

We aim to complement the explanation of \ac{para} in the previous 
section with a performance evaluation compared to the existing method
and a quantitative comparison to \ac{pwl}.
Our implementation is conducted in \software{Python 3.11.8} and can be found
here: \url{https://github.com/adriangoess/pwl-vs-para}.
For the shift in Method~\labelcref{alg:main} 
Step~\labelcref{alg:main:step:shift},
we use a convex combination $t_k \gets (1 - \lambda) t_{k-1} + \lambda t_k$,
where $\lambda = 0.9$.

\subsubsection{Comparison to Existing Method}

To the best of our knowledge, the only method achieving the same
type of \ac{para} approximation is presented in~\cite{goess2025para} and
has been described in~\Cref{subsec:para-model}.
Following the original source, we refer to the method as \texttt{practical-MIP}.
The method is leveraged to compute
\ac{para} approximations on different domains and varying tolerance
for different types of functions, including the sine and the exponential
function.
For a successful run, the resulting number of parabolas is reported.
However, there also exist cases for which the
\texttt{practical-MIP} approach was not able to generate a result.

We compare the number of (un)successful runs
and the number of resulting parabolas.
For $\sin$ and $\exp$, we consider the same intervals and tolerances as
in~\cite{goess2025para}, which are summarized in~\Cref{tab:para-approx-domains}.

\begin{table}
    \centering
    \begin{tabular}{lcccccc}
        \toprule
        & \multicolumn{6}{c}{domain $\domain$} \\ \cline{2-7}
        \\[-1em]
        $\exp$ &$[-5, -2]$ & $[-2,\, 2]$ & $[\, -5,\ \ 2]$& $[\phantom{-}2, 5]$& $[-2, 5]$ & $[-5, \phantom{\pi}5]$  \\ 
        $\sin$  & $[-\frac{\pi}{2}, \ \frac{\pi}{2}]$ & $[\,\frac{\pi}{2}, \frac{3\pi}{2}]$ &$[-\frac{\pi}{2}, \frac{3\pi}{2}]$ & $[\phantom{-}0, \pi]$ & $[\pi,\, 2\pi]$ & $[\phantom{-}0, 2\pi]$ \\
        \bottomrule
    \end{tabular}
    \caption{Function-domain combinations to approximate 
        with Method~\labelcref{alg:main} for comparison 
        with~\cite{goess2025para}.}
    \label{tab:para-approx-domains}
\end{table}

For the comparison, we leverage a similar type of table as used
in~\cite{goess2025para}.
The results regarding the sine function are then
summarized in~\Cref{tab:comparison-number-paras-sin}.
We denote the number of parabolas computed by \texttt{practical-MIP} 
-- if any -- in brackets and highlight the smaller such number in bold.

First of all, for the cases where either method is successful,
we observe mixed results.
There is no consistently superior method in terms of less parabolas, but in some
cases  Method~\labelcref{alg:main} finds a significantly lower number.
For a potential explanation of such a phenomenon, we have to take into account
that \texttt{practical-MIP} relies on a binary search in this quantity.
That is, after every run that reports the \ac{MIP} problem to be infeasible or
hits the time limit, the search procedure doubles the number of parabolas
and restarts.
Hence, it may return high intermediate numbers as they are the best known
when terminating after a fixed number of tries.

For the other cases, we see that Method~\labelcref{alg:main} is able to return
an approximating set of parabolas in all cases, for which \texttt{practical-MIP}
has failed, especially refer to the columns with a tolerance of $10^{-3}$.
The generality of the setup in \texttt{practical-MIP} allows for wide
application, but the underlying \ac{mip} problem requires a number of (binary)
variables growing linearly in the
tolerance and the global Lipschitz constant.
The resulting \ac{mip} problems may not be solved in a considerable
time limit, as specified in~\cite{goess2025para}.
Method~\labelcref{alg:main}, in contrast, can make use of
more restrictive properties like continuous differentiability
(when computing the maxima in~\labelcref{eq:inner-loop-maxima})
to return approximations on larger domains.
Surely, a tighter tolerance requires more parabolas and thus more computations.
However, the respective effort seems to be manageable.

\begin{table}
    \centering
      \begin{tabular}{ccccccccccc}\toprule
	& & \multicolumn{9}{c}{$\eps$} \\ \cline{3-11}
	& & \multicolumn{9}{c}{\vspace{-0.28cm}} \\
$\domain$ & & $10^{0}$ & $10^{-1}$ & $10^{-2}$ & $10^{-3}$  & & $10^{0}$ & $10^{-1}$ & $10^{-2}$ & $10^{-3}$ \\ \cline{3-6}\cline{8-11}
    \\[-1em] 
	$[-\pi/2, \phantom{3}\pi/2]$ &  & 1 (1) & 3 (3) & \textbf{7} (9) & \textbf{22} &  & 1 (1) & 3 (3) & \textbf{7} (8) & \textbf{22} \\ 
    \\[-1em] 
	$[\phantom{-}\pi/2, 3\pi/2]$ &  & 1 (1) & 3 (3) & \textbf{7} (8) & \textbf{22} &  & 1 (1) & 3 (3) & \textbf{7} (8) & \textbf{22} \\ 
    \\[-1em] 
	$[-\pi/2, 3\pi/2]$ &  & 2 (\textbf{1}) & 5 (5) & \textbf{14} (47) & \textbf{44} &  & 1 (1) & 5 (\textbf{3}) & 17 (\textbf{13}) & \textbf{51} \\ 
    \\[-1em] 
	$[0, \phantom{2}\pi]$ &  & 1 (1) & 2 (\textbf{1}) & 5 (5) & \textbf{16} (40) &  & 1 (1) & 1 (1) & 5 (\textbf{3}) & 16 (16) \\ 
    \\[-1em] 
	$[\pi, 2\pi]$ &  & 1 (1) & 1 (1) & 5 (\textbf{3}) & \textbf{16} (32) &  & 1 (1) & 2 (\textbf{1}) & 5 (5) & \textbf{16} \\ 
    \\[-1em] 
	$[0, 2\pi]$ &  & 2 (2) & \textbf{4} (5) & \textbf{14} (24) & \textbf{44} &  & 2 (2) & \textbf{4} (5) & \textbf{14} (44) & \textbf{44} \\ 
    \\[-1em] 
    & & \multicolumn{4}{c}{above} & & \multicolumn{4}{c}{below}\\\bottomrule
  \end{tabular}
	\caption{Number of parabolas to approximate $\sin$ with Method~\labelcref{alg:main} (\texttt{practical-MIP}).}

    \label{tab:comparison-number-paras-sin}
\end{table}

For the exponential function, both effects are observed in even greater extent,
compare~\Cref{tab:comparison-number-paras-exp}.
For cases, where both methods are successful, Method~\labelcref{alg:main}
returns at most as much approximating parabolas as \texttt{practical-MIP}.
That is, it outperforms \texttt{practical-MIP} on every domain-tolerance
combination.
In the other cases, \texttt{practical-MIP} struggles on domains with a large 
upper bound, since this implies a large global Lipschitz constant
as explained in~\cite{goess2025para}.
The obstacle seems to be circumvented with Method~\labelcref{alg:main}.
Again, this results from avoiding to solve several (large) \ac{mip} problems,
but instead only computing maxima on simple intervals.

\begin{table}
	\centering
	  \begin{tabular}{ccccccccccc}\toprule
	& & \multicolumn{9}{c}{$\eps$} \\ \cline{3-11}
	& & \multicolumn{9}{c}{\vspace{-0.28cm}} \\
$\domain$ & & $10^{0}$ & $10^{-1}$ & $10^{-2}$ & $10^{-3}$  & & $10^{0}$ & $10^{-1}$ & $10^{-2}$ & $10^{-3}$ \\ \cline{3-6}\cline{8-11}
    \\[-1em] 
	$[-5, -2]$ &  & 1 (1) & 1 (1) & 2 (2) & 5 (5) &  & 1 (1) & 1 (1) & 2 (2) & \textbf{5} (8) \\ 
    \\[-1em] 
	$[-2, \phantom{-}2]$ &  & 2 (2) & 5 (5) & \textbf{15} (32) & \textbf{47} &  & 2 (2) & 4 (4) & \textbf{13} (20) & \textbf{39} \\ 
    \\[-1em] 
	$[-5, \phantom{-}2]$ &  & 3 (3) & \textbf{7} (8) & \textbf{23} (81) & \textbf{70} &  & 2 (2) & 6 (6) & \textbf{16} (56) & \textbf{51} \\ 
    \\[-1em] 
	$[\phantom{-}2, \phantom{-}5]$ &  & \textbf{6} & \textbf{16} & \textbf{50} & \textbf{158} &  & 5 (5) & \textbf{14} (20) & \textbf{44} & \textbf{137} \\ 
    \\[-1em] 
	$[-2, \phantom{-}5]$ &  & \textbf{10} & \textbf{31} & \textbf{99} & \textbf{313} &  & \textbf{8} & \textbf{23} & \textbf{72} & \textbf{225} \\ 
    \\[-1em] 
	$[-5, \phantom{-}5]$ &  & \textbf{13} & \textbf{39} & \textbf{122} & \textbf{382} &  & \textbf{9} & \textbf{26} & \textbf{79} & \textbf{251} \\ 
    \\[-1em] 
    & & \multicolumn{4}{c}{above} & & \multicolumn{4}{c}{below}\\\bottomrule
  \end{tabular}
	\caption{Number of parabolas to approximate $\exp$ with Method~\labelcref{alg:main} (\texttt{practical-MIP}).}

	\label{tab:comparison-number-paras-exp}
\end{table}

\subsubsection{Number of Parabolas and Linear Pieces}
\label{subsubsec:number-para-lins}

Besides a computational comparison of the \ac{para} and the
\ac{pwl} relaxation technique on \ac{minlp} problems in terms of run time,
one might also wonder about the
number of approximates -- parabolas or linear pieces.
In a relaxation each such requires an additional constraint
and, in case of \ac{pwl}, additional variables.
We aim to answer the question computationally,
considering a fixed tolerance of $\eps=0.1$ and varying domains for fixed
functions.
The computation of the linear pieces follows the explanation
in~\Cref{subsec:pwl-model} and is performed by slightly modifying the
code provided in~\cite{braun2023pwl-t-rex}.
As explained, it is run with a tolerance of $\eps/2$ and the returned
approximation is then shifted by $\eps/2$ to receive a true relaxation.
For computing the parabolas, we use Method~\labelcref{alg:main}.

At first, we take $f = \sin$ into account and investigate the number of
approximates on $[0, \ell\pi]$ for $\ell \in \{1, 2, 3\}$.
For \ac{para}, we get 1, 6, 8, parabolas, respectively, which is
visualized in~\Cref{fig:para-approx-sin-varying-domain}.
For \ac{pwl}, we end up with 4, 8, 12 linear pieces as shown 
in~\Cref{fig:pwl-approx-sin-varying-domain}.
Such a linear growth of the number of pieces is to be expected
when considering a periodic function like $\sin$.

\begin{figure}
    \centering
    \includegraphics[height=3.3cm,keepaspectratio]{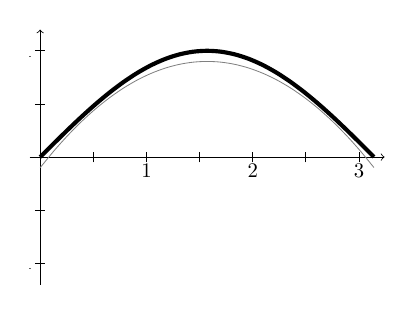}
    \includegraphics[height=3.3cm,keepaspectratio]{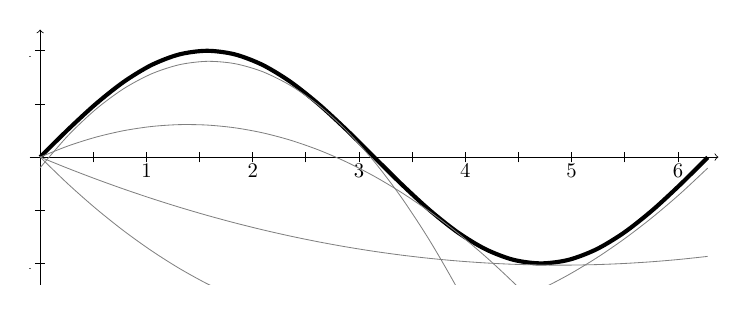}\\
    \includegraphics[height=3.3cm,keepaspectratio]{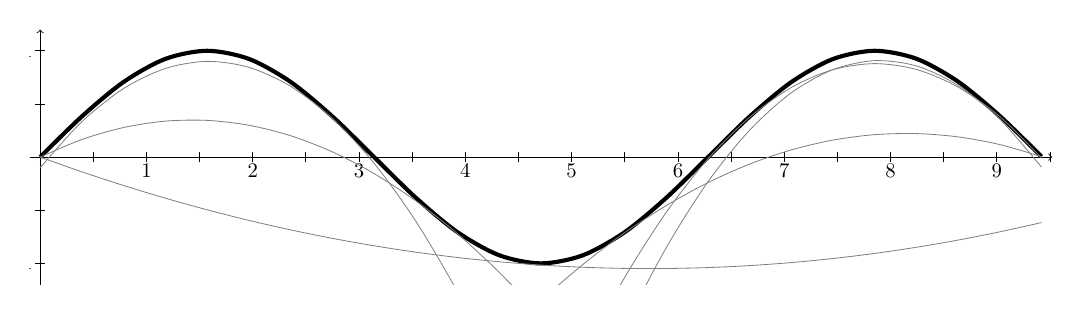}
    \caption{\ac{para} approximations of $\sin$ with $\eps=0.1$ on 
        $[0, \ell\pi]$ for ${l \in \{1, 2, 3\}}$ by 1, 4, 6 parabolas,
        respectively.}
    \label{fig:para-approx-sin-varying-domain}
\end{figure}

\begin{figure}
    \centering
    \includegraphics[height=3.3cm,keepaspectratio]{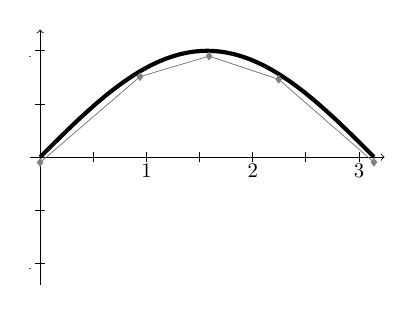}
    \includegraphics[height=3.3cm,keepaspectratio]{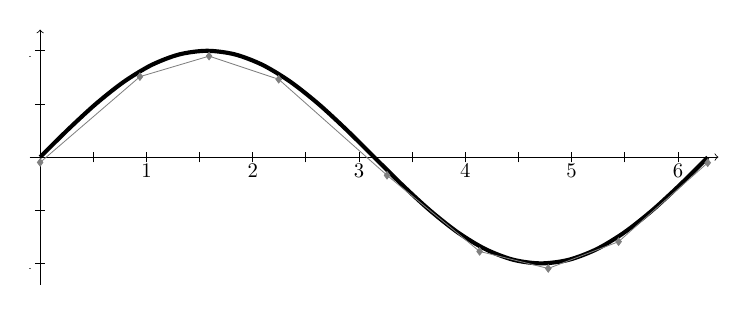}\\
    \includegraphics[height=3.3cm,keepaspectratio]{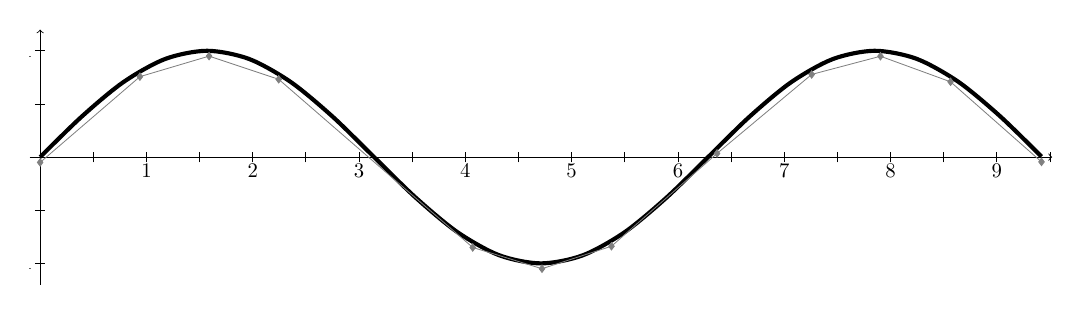}
    \caption{\ac{pwl} approximations of $\sin$ on with $\eps=0.1$ on
        $[0, \ell\pi]$ for ${l \in \{1, 2, 3\}}$ with 4, 8, 12 pieces,
        respectively.}
    \label{fig:pwl-approx-sin-varying-domain}
\end{figure}

Contemplating the approximations on $[0, \pi]$, the $\sin$ function
is reminiscent of a concave parabola and thus requires a lower number of such.
However, for the convex part, \ac{para} needs to respect
the global underestimation~\labelcref{eq:global-under} and 
accordingly a larger domain.
Hence, it has not been obvious to us that a lower number of parabolas
is sufficient even when enlarging the respective domain.
In conclusion, \ac{para} seems to approximate the sine
function with quite a few approximating parabolas,
requiring even less approximates than \ac{pwl}.

In the search for other examples, we deem a function that combines high slope
and an asymptotically linear part as a challenge for \ac{para}.
Hence, we consider the natural logarithm $\ln$ on a variety of domains,
in particular $[e^{-4}, e^{2\ell}]$ for $\ell \in \{-1, 0, 1\}$.
The corresponding illustrations for the \ac{para} and \ac{pwl}
approximations can be found in~\Cref{fig:para-approx-ln-varying-domain}
and~\Cref{fig:pwl-approx-ln-varying-domain}, respectively.

\begin{figure}
    \centering
    \includegraphics[height=3.3cm,keepaspectratio]{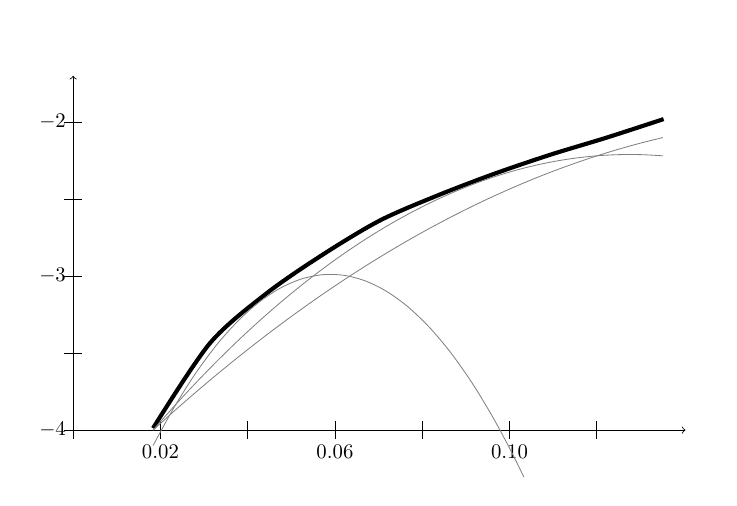}
    \includegraphics[height=3.3cm,keepaspectratio]{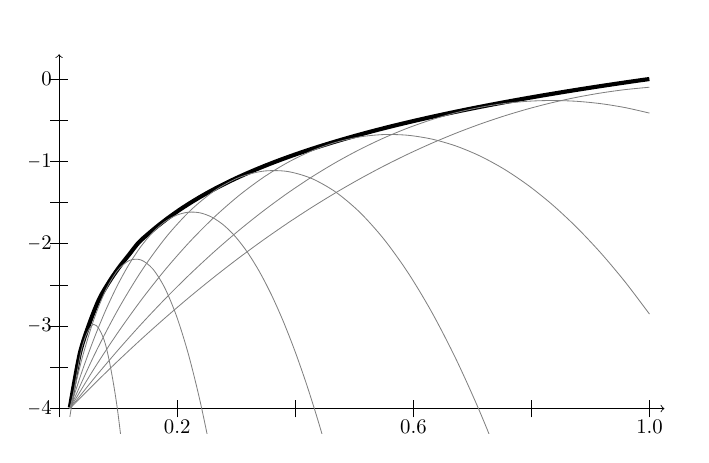}\\
    \includegraphics[height=3.3cm,keepaspectratio]{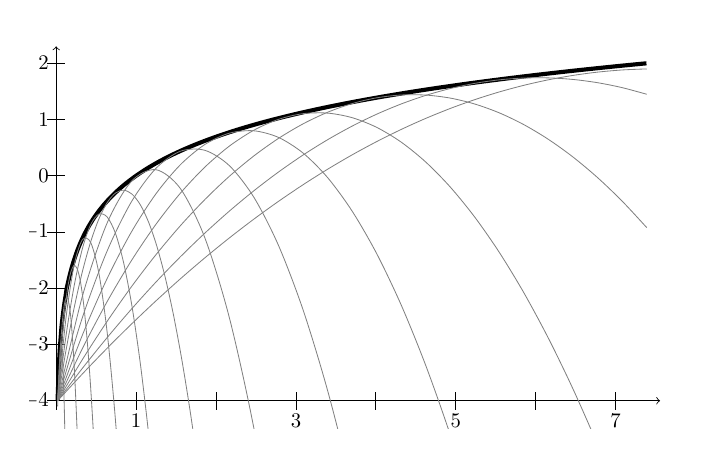}
    \caption{\ac{para} approximations of $\ln$ with $\eps = 0.1$ on 
       $[e^{-4}, e^{2\ell}]$ for ${l \in \{-1, 0, 1\}}$ with 3, 7, 13 parabolas,
       respectively.}
    \label{fig:para-approx-ln-varying-domain}
\end{figure}

\begin{figure}
    \centering
    \includegraphics[height=3.3cm,keepaspectratio]{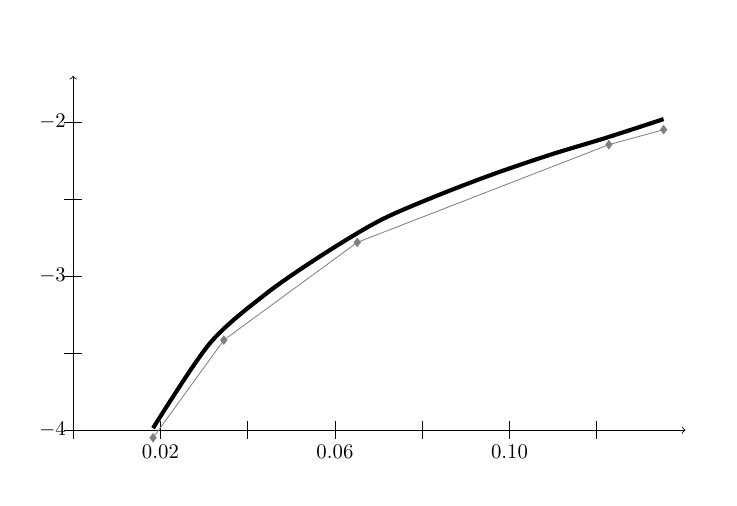}
    \includegraphics[height=3.3cm,keepaspectratio]{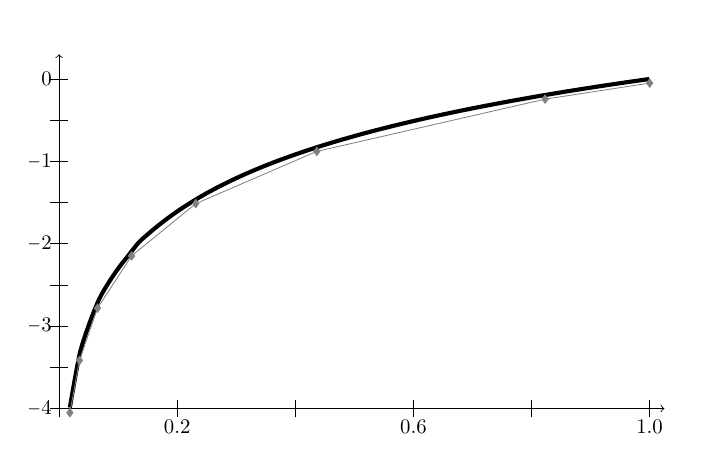}\\
    \includegraphics[height=3.3cm,keepaspectratio]{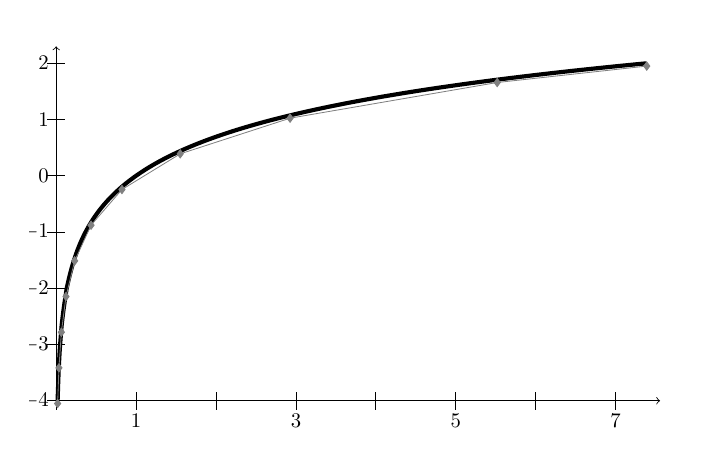}
    \caption{\ac{pwl} approximations of $\ln$ with $\eps = 0.1$ on 
        $[e^{-4}, e^{2\ell}]$ for ${l \in \{-1, 0, 1\}}$ with 4, 7, 10 pieces,
        respectively.}
    \label{fig:pwl-approx-ln-varying-domain}
\end{figure}

As before, \ac{para} requires less approximates for the smallest interval, \ie,
3~parabolas versus 4~linear pieces.
However, when we enlarge the function domain, the number of parabolas
(7 and 13) increasingly exceeds the number of linear pieces (7 and 10).
Indeed, this results from the described form of $\ln$ and the property of
globally underestimating the function with every parabola on the entire domain.

In conclusion, those two examples showcase that there is no general statement
about the necessary number of linear pieces with respect to the parabolas.
Nonetheless, the investigation provides an intuition about it.

\section{Comparison of Relaxation Techniques}
\label{sec:relaxation-comparison}

Using the \ac{para} method from the previous section, 
we are enabled to compute \ac{para} relaxations for general nonlinear,
one-dimensional constraint functions on relatively large domains and
up to small approximation tolerances.
It allows for a comparison with the \ac{pwl} relaxation technique.
In particular, we construct $\eps$-relaxations in the form
of~\labelcref{eq:problem-relax}
for varying values of the approximation tolerance $\eps$ using either technique.
If a solver can compute an optimal solution to such a relaxation,
this solution is also optimal for the 
corresponding problem~\labelcref{eq:problem-approx}.
Hence, we can simply compare the run times achieved with either relaxation,
fixing a tolerance~$\eps$.
The code can again be found in \url{https://github.com/adriangoess/pwl-vs-para}.

\subsection{Test set}
For a potential test set, we consider the MINLPLib~\cite{bussieck2003minlplib}
with 1595 \ac{minlp} instances.
First of all, we aim to leverage the expression tree structure explained
in~\Cref{subsec:math-problem} and thus restrict to the instances available in
\software{OSiL} format~\cite{fourer2010osil}, which are all but one.
Second, we only consider instances that feature at least one exponential,
natural logarithm, sine, or cosine function, since we focus on the their
relaxation.
We remain with an intermediate subset of 235 instances.

Besides this formal filtering we need to respect the characteristics of the
relaxation techniques.
Particularly, either one requires finite bounds for the constraint function
to relax and the size of the resulting relaxations, \ie, the number of linear
pieces or parabolas, depends on the size of the finite domain.
Typically, the function domains in such instances are sterile,
meaning that a part of the domain is not even feasible, considering
function properties or the relationships in the overall problem.
That is, we can tighten the domains by applying a presolving step.
For this, we make use of \software{SCIP 10.0}~\cite{hojny2025scip10}.
It already incorporates the functionality of pure presolving which we apply
with the settings \software{constraints/and/linearize=TRUE} and
\software{presolving/gateextraction/maxrounds=0}.
These ensure to avoid so-called ``and''-constraints in the resulting problems,
which is necessary to write them in respective format.

With respect to format, \software{SCIP} has a reader-functionality for
\software{OSiL}, but no writer-functionality, so we write in \software{GAMS}
format and re-convert to \software{OSiL} using
\software{GAMS 51.3.0}~\cite{gams51}.
Hereby, we ruled out two instances containing the entropy function, since they
could not be converted accordingly.
In addition, we had to rule out five instances where this conversion step failed
and we could not resolve its issue.

Despite the presolving step, however, there remain instances featuring
general nonlinear constraint functions defined on infinite domains or domains
that we consider numerically challenging.
Specifically, the latter may either contain very large absolute values and/or
force the corresponding function to be very close to zero or tremendously large.
This can lead to numerical errors when constructing the relaxations or even
solving such problems.
Hence we restrict to instances that show the domain of an exponential function
to be a subset of $[-5 \cdot 10^5, 10]$ and the domain of a natural logarithm to
 be a subset of $[10^{-7}, 10^7]$.
We remain with a final test set of 174 instances.

\subsection{Computational Setup}

For all instances in our test set, we analyze the occurring general nonlinear,
one-dimensional functions with their presolved domains.
For computing \ac{pwl} approximations, we consider each 
function-domain combination and compute a respective
relaxation as detailed in~\Cref{subsec:pwl-model} using a slightly
manipulated version of the code provided in~\cite{braun2023pwl-t-rex}.
The manipulation consists of relaxing only the general nonlinear functions
and writing the respective problem in \software{GAMS} format.

For \ac{para}, we stick to the paradigm introduced in~\cite{goess2025para},
which advocates the use of \ac{para} approximations in terms of a look-up table.
That is, we cluster the bounds obtained in the presolved instances for the
general nonlinear functions and compute the corresponding parabolas beforehand.
In particular, for the exponential function, we round a lower bound down to the
next multiplicative order of magnitude if its less than -1, 
else to the first digit.
This ensures to cluster lower bounds with similar function value, \ie, small
lower bounds, into the same subintervals, while respecting larger differences
with increasing bound.
To give an example, -132 is rounded to -200, whereas -0.456 to -0.5.
For the upper bound, we proceed similarly for the same reason, but round it up
and consider the
first digit if it is less or equal 0, but the second digit if it is positive.
For the natural logarithm, a reversed procedure is conducted:
For a bound between $10^{\ell - 1}$ and $10^{\ell}$, $l \in \integers$,
it is rounded to the next multiplicative of $10^{\ell -1}$.
This is capped for the lower bound with $\ell \leq 3$ and for the upper bound
with $\ell \geq -1$, where the corresponding multiplicative of $\ell = 3$ and
$\ell = -1$ is considered, respectively.
For the sine and cosine function, the bounds are simply grouped in
multiplicatives of 0.1 for domains with a length of less or equal $4\pi$,
otherwise we concatenate an approximation on an entire period by shifting.

We perform the upper procedure for both relaxation techniques for varying
approximation tolerance $\eps \in \{10^0, 10^{-1}, 10^{-2}, 10^{-3}, 10^{-4}\}$.
Each problem is written in \software{GAMS} format and converted as explained
above.
The resulting problems are solved by leveraging \software{SCIP 10.0} once
again.
All those solving processes are carried out on single nodes with Intel Xeon Gold
6326 ``Ice Lake'' multicore processors with 2.9\unit{\giga\hertz} per core.
The accessible RAM per node is set to 32\unit{\giga\byte}, where \software{SCIP}
is restricted to only use 30\unit{\giga\byte}, avoiding issues with
reading/modeling overhead.
For every problem, we impose a time limit of 4\unit{\hour} equaling
14400\unit{\second}.

Note that there occurred issues in conversion as well as memory issues and
numerical troubles.
If one or both resulting relaxations were affected by one such, we removed
them from further consideration.
This decreased the number of instances for the comparison to 150.

\subsection{``Clash'' of Relaxations}

As laid out before, we can restrict our comparison to sole run time.
For this, we start by comparing the number of faster runs per tolerance $\eps$.
Two runs on the same instance are considered equally fast, if the solution time
varies by less than 5\unit{\second}.
\Cref{fig:faster-solve-comparison} depicts the number of instances for which
the respective relaxation approach could be solved in less time,
for decreasing tolerance from left to right.
The lines indicate trends.

\begin{figure}
    \includegraphics[width=\textwidth]{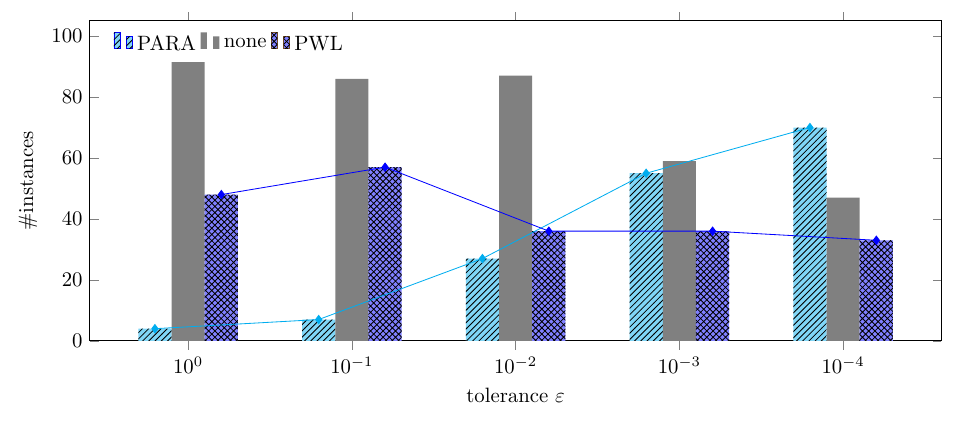}
    \caption{Number (\#) of instances solved faster by \ac{para} or \ac{pwl}
        per tolerance $\eps$.}
    \label{fig:faster-solve-comparison}
\end{figure} 

We observe that the number of instances without a clear winner is comparably
high, decreasing with lower approximation tolerance.
On the one hand, this is an effect by hard instances on which both relaxations
hit the time limit throughout all tolerances~$\eps$.
On the other hand, small-sized instances require quite simple relaxations and
are thus solved in a short amount of time for either relaxation.
The latter effect is reduced with very small tolerance, since it requires a
considerable number of parabolas or linear pieces, causing the solving time
to deviate significantly.

For the comparison of \ac{para} and \ac{pwl}, we see a great favor towards the
\ac{pwl} approach for large tolerances.
In those cases, \ac{pwl} requires none to only a few additional binary variables
along the linear constraints and thus seems to be more tractable than,
potentially non-convex, quadratic constraints.
For smaller tolerance, however, the advantage of not introducing additional
binary variables seems to outweigh the nonlinearity of the parabolic
-- thus quadratic -- constraints.
The internal expressions in \software{SCIP} seem to handle additional
similar-shaped constraints quite well in comparison.

In order to investigate this phenomenon to greater extent, we also evaluate
the results in form of performance profiles using the software
of~\cite{siqueira2016perprof}.
Such performance profiles compare the the run time of the solver on a specific
relaxation to the fastest/shortest run time.
In particular, they depict the fraction of problems solved within a factor
(the performance ratio) of the fastest run available.
For a fair comparison, we also include a base run, \ie, solving the
presolved instances without reformulation or relaxation.
The profiles are showcased for the extreme and intermediate cases
$\eps \in \{10^0,\,10^{-2},\,10^{-4}\}$ in~\Cref{fig:perprof-1e0},
\Cref{fig:perprof-1e-2}, and \Cref{fig:perprof-1e-4}, respectively.

\begin{figure}
    \includegraphics[scale=0.9]{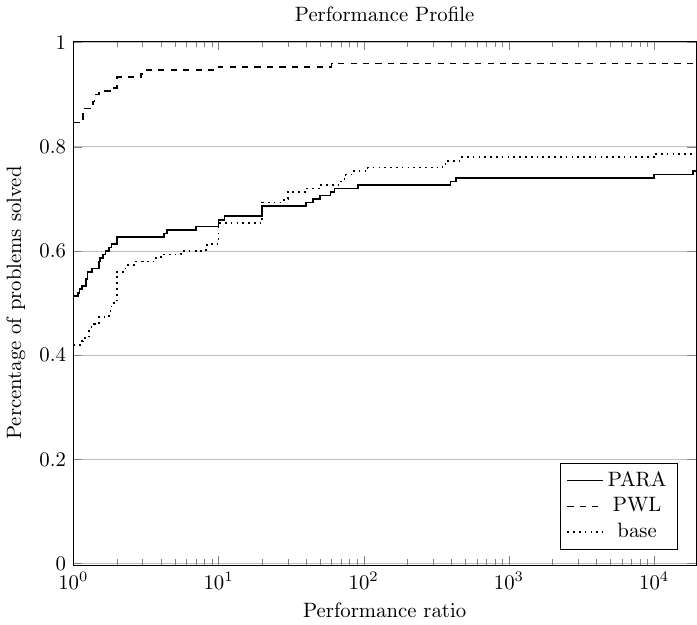}
    \caption{\ac{para} and \ac{pwl} for $\eps = 10^0$ and base runs.}
    \label{fig:perprof-1e0}
\end{figure}

For the highest tolerance $\eps = 10^0$ in \Cref{fig:perprof-1e0},
we can see the aforementioned strength of \ac{pwl}.
Few binary variables and linear constraints are tractable in modern solvers.
Interestingly, the introduction of \ac{para} constraints is only a benefit up
to a certain ratio and already falls behind later on.
Such effects can be caused by approximating already convex constraints,
\eg, an underestimation of the exponential function, by multiple quadratic
constraints.
Note that this was not ruled out a-priorily to allow the comparison against
\ac{pwl}.

\begin{figure}
    \includegraphics[scale=0.9]{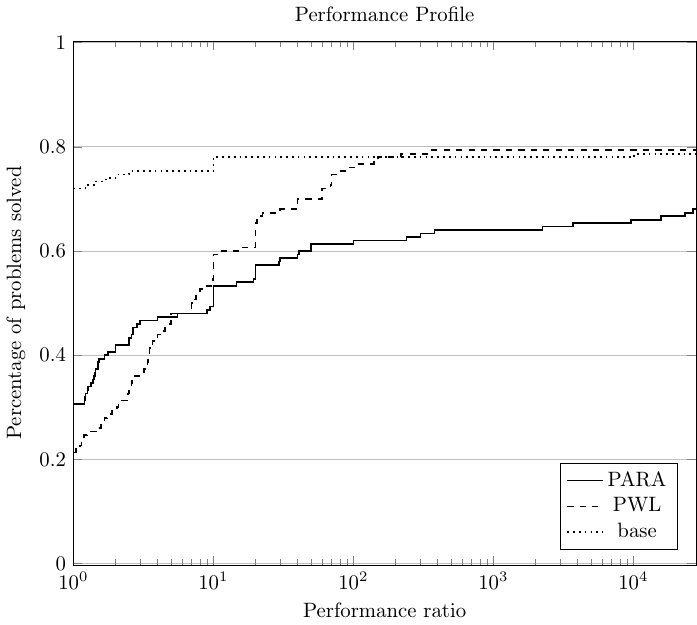}
    \caption{\ac{para} and \ac{pwl} for $\eps = 10^{-2}$ and base runs.}
    \label{fig:perprof-1e-2}
\end{figure}

\begin{figure}
    \includegraphics[scale=0.9]{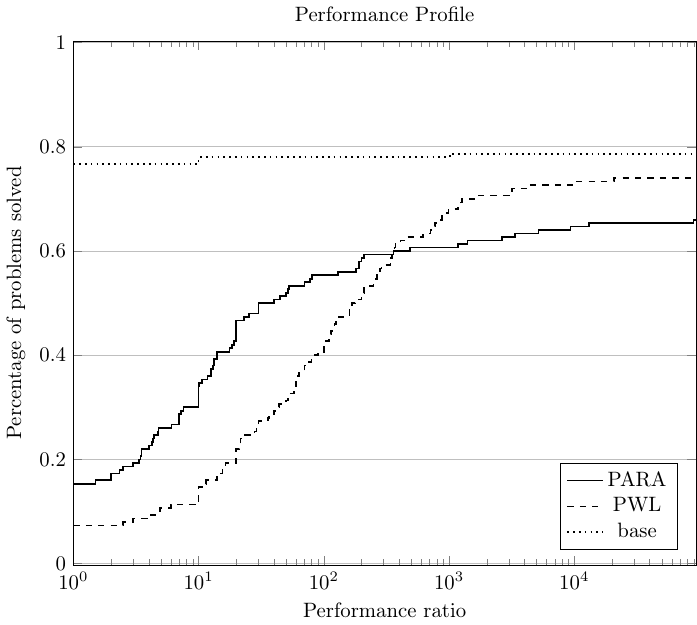}
    \caption{PARA and PWL for $\eps = 10^{-4}$ and base runs.}
    \label{fig:perprof-1e-4}
\end{figure}

For $\eps = 10^{-2}$ and $\eps = 10^{-4}$, we first observe
that the base run is almost always the fastest approach, see
\Cref{fig:perprof-1e-2} and \Cref{fig:perprof-1e-4}, respectively.
This may be a consequence by the lack of adaptivity for the presented
techniques.
In particular, a solver refines the underlying relaxation locally to
avoid large overhead and thus adaptively incorporates additional relaxation
constraints.
This is the default for piecewise linear relaxations in modern solvers, yet
there exists no implementation for the parabolic approach.
Since the latter requires new theoretical foundations and implementations,
we consider it as out of scope for the current paper and take up on it in
\Cref{sec:conclusion}.

Note that the base run is included to provide an insight about
the original difficulty of the instances.
However, the focus of this article truly lies in a comparison of the relaxation
techniques in terms of computational efficacy.

Turning to this comparison of \ac{para} and \ac{pwl} in those cases, we can see
that for $\eps = 10^{-2}$, \ac{para} already shows a better performance for
ratios up to four, before \ac{pwl} takes over.
For $\eps = 10^{-4}$ this is even increased up until around a ratio of 300,
which is quite high.
These observations reinforce the effect described above:
For small tolerances, (non-convex) quadratic constraints can be leveraged more
efficiently than additional binary variables and linear pieces, but only up to a
certain point.

In the light of \Cref{subsubsec:number-para-lins}, instances that feature sine
or cosine functions and not the natural logarithm, maybe be especially
beneficial for the \ac{para} approach.
Such problems occur in the optimization of power networks,
see~\cite{aigner2023acopf-discrete-decisions,bienstock2020acopf,
    frank2016opf-introduction}, for instance.
Hence, we consider a subset containing only instances with at least one sine
or cosine function, which comprises exactly those instances without a natural
logarithm nonlinearity (on our test set).
This results in a subset of 33 instances.

We conduct the same comparison as before and display the number of instances
with faster runs using the respective approach in
\Cref{fig:faster-solve-comparison-no-ln}.
As expected, the previously detected effects are enhanced.
Relaxations by PARA are solved faster for all but one instances for tolerances
$\eps = 10^{-2}$ and below.

\begin{figure}
    \includegraphics[width=\textwidth]{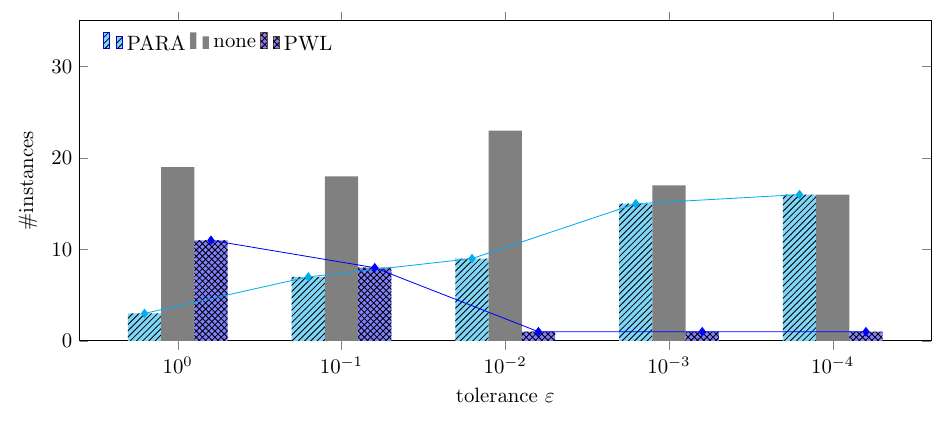}
    \caption{Number (\#) of instances without $\ln$ solved faster by \ac{para}
        or \ac{pwl} per tolerance $\eps$.}
    \label{fig:faster-solve-comparison-no-ln}
\end{figure}

Turning to the performance profile for $\eps = 10^0$ in 
\Cref{fig:perprof-1e0-no-ln}, we notice that both relaxations return instances
being solved faster than the base ones.
Though \ac{pwl} relaxations still require less time than \ac{para}, the effect
is mitigated in comparison to the entire test set.

Now, for $\eps = 10^{-2}$, \ac{para} already consistently outperforms \ac{pwl}
and even is as least as fast as the base runs, not respecting any kind of
adaptivity, see \Cref{fig:perprof-1e-2-no-ln}.
Whereas the exceed in performance between the relaxation techniques
prevails for even smaller tolerance of $\eps = 10^{-4}$ in 
\Cref{fig:perprof-1e-4-no-ln}, we observe that up until a factor of 100,
the base runs are faster.
We dedicate this to the same effects mentioned above.

Suppose we try to solve instances that mainly feature (co)sine functions in
their constraints and require medium to small tolerances.
Then, the above results demonstrate that \ac{para} relaxations
can provide problem formulations that are being solved faster than tackling
the original problem with a stand-alone solver.
This makes \ac{para} relaxations particularly advisable for usage in fields
with power network optimization like AC-OPF
problems~\cite{aigner2023acopf-discrete-decisions,
    bienstock2020acopf,frank2016opf-introduction}.

\begin{figure}
    \includegraphics[scale=0.9]{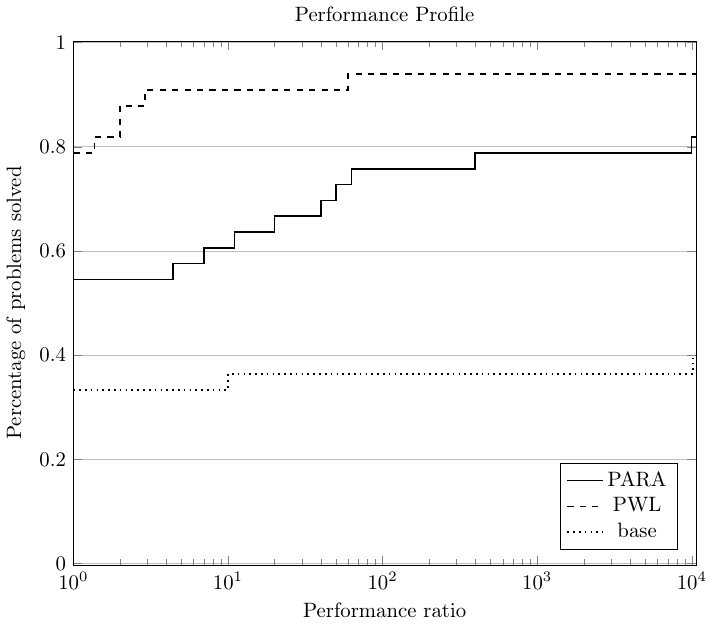}
    \caption{\ac{para} and \ac{pwl} for $\eps = 10^0$ and base runs on
        instances without $\ln$.}
    \label{fig:perprof-1e0-no-ln}
\end{figure}

\begin{figure}
    \includegraphics[scale=0.9]{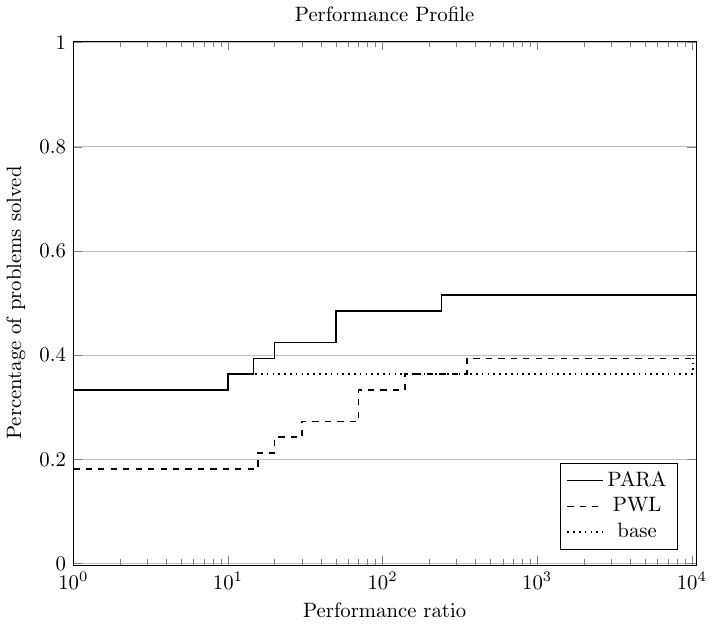}
    \caption{\ac{para} and \ac{pwl} for $\eps = 10^{-2}$ and base runs on
        instances without $\ln$.}
    \label{fig:perprof-1e-2-no-ln}
\end{figure}

\begin{figure}
    \includegraphics[scale=0.9]{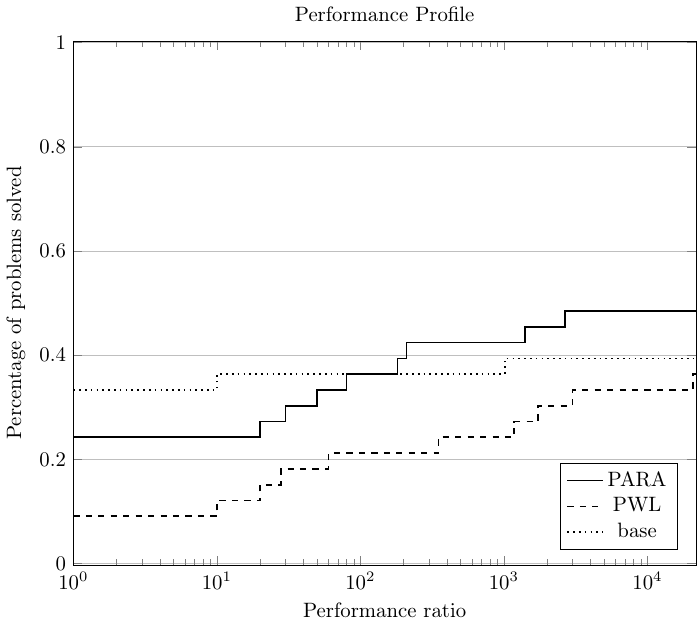}
    \caption{\ac{para} and \ac{pwl} for $\eps = 10^{-4}$ and base runs on
        instances without $\ln$.}
    \label{fig:perprof-1e-4-no-ln}
\end{figure}

Lastly, we'd only like to mention that the relaxation techniques were able to
provide dual bounds that improve on the ones reported and achieved during
base runs.
We report them and give a quick overview in~\Cref{tab:dual-gap-improvements}
in~\Cref{sec:dual-gap-improvements}.

\section{Conclusion}
\label{sec:conclusion}

We have introduced a novel method to compute \ac{para} approximations for
one-dimensional, continuously differentiable functions, which outperformed the
existing method in terms of success in formerly intractable cases and mostly
returned fewer parabolas.
The theoretical basis was complemented with a comparison to \ac{pwl}
with respect to number of linear pieces/parabolas:
Whereas \ac{para} approximations for the sine function require less approximates
than \ac{pwl} ones throughout different interval sizes,
for the natural logarithm \ac{pwl} seems to require less such with increasing
interval size.
This is accounted to the global underestimation property causing troubles
when there exists a steep slope converging asymptotically to zero.

The ``clash'' or comparison of \ac{para} and \ac{pwl} as \ac{minlp} relaxations
was performed on a subset of the well-established MINLPLib.
Hereby, \ac{pwl} relaxations seem to be particularly favorable for wider
tolerances, since they require only a few binary variables and linear
constraints.
This favor shifts towards \ac{para} for tighter tolerances, as the problem sizes
for \ac{pwl} grows too large, whereas many quadratic constraints without
additional variables appears to be quite tractable.
As before, there also exists a preference for trigonometric functions when using
\ac{para}, as they appear in power network optimization.

Although considering a fixed tolerance provides a valid insight for the usage
of either relaxation, practical solvers iteratively refine the given relaxation
until a satisfactory tolerance is reached.
For \ac{para}, this requires further theoretical developments and extended
implementation effort, thus being out of reach for the present article.
However, the adjustment of the presented method and a comparison with
adaptive \ac{pwl} relaxations appears to be a natural next step.

\begin{acronym}[MINLPs]
    \acro{minlp}[MINLP]{mixed-integer nonlinear programming}
    \acro{minlpp}[MINLP]{mixed-integer nonlinear program}
    \acro{minlpps}[MINLPs]{mixed-integer nonlinear programs}
    \acro{miqcp}[MIQCP]{mixed-integer quadratically-constrained programming}
    \acro{mip}[MIP]{mixed-integer linear programming}
    \acro{mipp}[MIP]{mixed-integer linear program}
    \acro{para}[PARA]{global parabolic}
    \acro{paralong}[PARA]{global one-sided parabolic}
    \acro{pwl}[PWL]{piecewise linear}
\end{acronym}
\section*{Acknowledgments}
\label{sec:acknowledgements}

We thank the Deutsche Forschungsgemeinschaft for their
support within project A05 in the
\enquote{Sonderforschungsbereich/Transregio 154 Mathematical
  Modelling, Simulation and Optimization using the Example of Gas
  Networks}, project ID 239904186.
We also thank Robert Burlacu, Alexander Martin, and Kristina Rolsing
for fruitful discussions about this topic and for reviewing an earlier draft.

\printbibliography
\newpage
\ 
\appendix
\section{Proofs for Correctness Theorems in \Cref{subsec:method}}
\label{sec:appendix-proofs}

In this section, we provide the detailed proofs for 
\Cref{thm:correctness-inner-loop} and \Cref{thm:correctness-outer-loop}.
For the sake of completeness, we restate them and provide the proof afterwards.

\subsection{Proof of \Cref{thm:correctness-inner-loop}}
\label{subsec:proof-inner-loop}

\thmone*

\begin{proof}
    First of all, note that the results in \Cref{lem:bounds-a} suggest that
    there exists a solution $(a, b, c)$ for~\labelcref{eq:k-th-problem}
    and~\labelcref{eq:b-c-computation-explicit} if and only if
    there exist finite bounds
    \begin{equation*}
        \lb a \geq \sup_{x \in \mathrm{int}(\locdomain)} A(x)
    \end{equation*} 
    and 
    \begin{equation*}
        \ub a \leq \min\left\{
        \inf_{x \in \domain\setminus\locdomain} A(x),
        \inf_{x \in \mathrm{int}(\locdomain)} B(x)
        \right\}
    \end{equation*}
    with $\lb a \leq \ub a$.
    In other terms, $\lb a > \ub a$ if and only if there exists no such 
    solution.
    
    Assume first that there exists a solution $(a, b, c)$.
    Further, let $l$ be the iteration index of the while-loop, \ie,
    the initialization sets $\ub{a}_0$, $\lb{a}_0 \gets -\infty$, and
    $a_0 \gets \ub{a}_0$
    and every update increments $l$.
    With this notation it holds true that $\lb{a}_l \leq \lb{a}$ and 
    $\ub{a}_l \leq \ub{a}$ for all $l \in \naturals_0$.
    In turn, since $a_l$ is always chosen to be $\ub{a}_l$, 
    it is $a_l \geq \lb{a}$ for all $l \in \naturals_0$ and thus the 
    condition~\labelcref{eq:k-th-problem-under} is always met on $\locdomain$.
    
    Method~\labelcref{alg:subroutine} has two potential outcomes:
    Either it returns a solution in Step~\labelcref{alg:subroutine:step:return}
    or it executes the while-loop infinitely often. 
    In the former case, the solution is valid for~\labelcref{eq:k-th-problem}
    by the definition of the maxima in~\labelcref{eq:inner-loop-maxima}, as well
    as satisfies~\labelcref{eq:b-c-computation-explicit} by construction.

    If Method~\labelcref{alg:subroutine} does not return a solution,
    it produces a sequence of points $(x_l)_{l \in \naturals_0}$ 
    that achieve the maximum of~\labelcref{eq:inner-loop-maximum-outside}
    and~\labelcref{eq:inner-loop-maximum-in-eps} and a sequence of $(a_l)_l$
    from evaluating the respective term in \Cref{lem:bounds-a} at $x_l$.
    
    Note that by using the equalities~\labelcref{eq:b-c-computation-explicit}
    we can write the parabola $p$ as parameterized in the 
    quadratic coefficient $a$.
    Specifically, subtraction and a rearrangement 
    of~\labelcref{eq:b-c-computation-explicit} gives
    \begin{equation*}
        b = \frac{f(\ub t) - f(\lb t)}{\ub t - \lb t} - a(\ub t + \lb t).
    \end{equation*}
    Plugging this back into either of the equalities and 
    rearranging for $c$ leads to
    \begin{equation*}
        c = f(\lb t) - \eps + \lb t \left(a \ub t - \frac{f(\ub t) - 
            f(\lb t)}{\ub t - \lb t}\right)
    \end{equation*}
    and 
    \begin{equation*}
        c = f(\ub t) - \eps + \ub t \left(a \lb t - \frac{f(\ub t) - 
            f(\lb t)}{\ub t - \lb t}\right),
    \end{equation*}
    respectively. 
    The parameterized version of a parabola $p$ is then given as
    \begin{equation}
        \begin{aligned}
            p(x; a) & = a (x^2 - (\ub t + \lb t) x + \lb t \ub t) + 
            (x - \lb t) \left(\frac{f(\ub t) - f(\lb t) }{\ub t - \lb t}\right) 
            + f(\lb t) - \eps \\
            & = a (x^2 - (\ub t + \lb t) x + \lb t \ub t) + 
            (x - \ub t) \left(\frac{f(\ub t) - f(\lb t) }{\ub t - \lb t}\right) 
            + f(\ub t) - \eps.
        \end{aligned}
        \label{eq:parameterized-parabola}
    \end{equation}
    The term in brackets after $a$ can also be written as 
    $q(x) \define (\ub t - x)(\lb t - x)$.
    Note that $q$ is a convex parabola symmetric to its global minimum at
    $(\ub t + \lb t)/2$ and thus attains its maximum as $\lb x$ or $\ub x$.
    We denote the corresponding function values as $q_{\min}$ and $q_{\max}$.
    Further, see that $q(x) > 0$ for $x \in \domain\setminus\locdomain$,
    whereas $q(x) < 0$ for $x \in \mathrm{int}(\locdomain)$.
    Combined with the fact that $a_l$ is non-increasing throughout the method,
    \ie, $a_l \leq a_{l-1}$, this representation allows to derive 
    \begin{equation}
        \label{eq:decreasing-param-p}
        p(x; a_l) \leq p(x; a_{l-1}), \qquad \text{for } 
        x \in \domain\setminus\locdomain,
    \end{equation}
    and
    \begin{equation}
        \label{eq:increasing-param-p}
        p(x; a_l) \geq p(x; a_{l-1}), \qquad \text{for } 
        x \in \mathrm{int}(\locdomain).
    \end{equation}
    
    Now, we aim to show that 
    \begin{enumerate}
        \item[(a)] for all $x \in \domain\setminus\locdomain$, 
        $\lim\limits_{l \to \infty} p(x; a_l) - f(x) \leq 0$, and
        \item[(b)] for all $x \in \locdomain$, 
        $\lim\limits_{l \to \infty} p(x; a_l) -f(x) + \varepsilon \geq 0$.
    \end{enumerate}
    
    We start with case (a).
    Without loss of generality, $\domain\setminus\locdomain \neq \emptyset$,
    otherwise we are done.
    It follows that $\lb x < \lb t$ or $\ub t < \ub x$, thus $q_{\max}  > 0$.
    
    We aim to prove the case by contradiction and thus
    assume there exists $\tilde{x} \in \domain\setminus\locdomain$ such that
    \begin{equation*}
        p(\tilde{x}; a_l) - f(\tilde{x}) > \nu, 
        \qquad \text{for all } l \in \naturals_0,
    \end{equation*}
    for a $\nu > 0$.
    Let $x_l^1$ and $x_l^3$ denote the optimal points 
    for~\labelcref{eq:inner-loop-maximum-outside} and 
    \labelcref{eq:inner-loop-maximum-in-eps} in iteration~$l$, respectively.
    Then, $a_{l+1} = \ub{a}_{l+1} = \min\{A(x_l^1), B(x_l^3)\}$ and
    by \Cref{lem:bounds-a}(a) it follows that $p(x_l^1;a_{l+1}) \leq f(x_l^1)$.
    
    In turn, this implies that $(x_l^1)_{l\in \naturals_0}$ is a sequence of 
    distinct points and, as they represent the optimal point, it is
    \begin{equation*}
        p(x_l^1; a_l) - f(x_l^1) \geq p(\tilde{x}; a_l) - f(\tilde{x}) > \nu,
        \qquad \text{for all } l \in \naturals_0.
    \end{equation*}
    Combining both statements, we have
    \begin{align*}
        \nu & < p(x_l^1; a_l) - f(x_l^1) \leq 
        p(x_l^1; a_l) - f(x_l^1) - (p(x_l^1; a_{l+1}) - f(x_l^1)) \\
        & = p(x_l^1; a_l) - p(x_l^1; a_{l+1}) = (a_l - a_{l+1})q(x_l^1) 
        \leq (a_l - a_{l+1})q_{\max}.
    \end{align*}
    A rearrangement gives $a_{l+1} < a_l - \nu/q_{\max}$, which
    implies $a_l \to -\infty$ for $l \to \infty$, when applying induction.
    This is a contradiction to $-\infty < \lb a \leq \ub a \leq a_l$ 
    for all $l \in \naturals_0$.
    
    For case (b), we can proceed analogously.
    So, $\mathrm{int}(\locdomain) \neq \emptyset$ without loss and 
    we thus know $q_{\min} < 0$.
    We assume there exists $\tilde{x} \in \mathrm{int}(\locdomain)$ such that
    \begin{equation*}
        p(\tilde{x}; a_l) - f(\tilde{x}) + \eps < -\nu, 
        \qquad \text{for all } l \in \naturals_0,
    \end{equation*}
    for a $\nu > 0$.
    By~\Cref{lem:bounds-a}(c) and the statements from before
    we can follow that $p(x_l^3;a_{l+1}) \geq f(x_l^3) - \eps$ for
    $l \in \naturals_0$.
    In addition, we derive
    \begin{equation*}
        f(x_l^3) - p(x_l^3; a_l) - \eps 
        \geq f(\tilde{x}) - p(\tilde{x}; a_l) - \eps > -\nu, 
        \qquad \text{for all } l \in \naturals_0.
    \end{equation*}
    A combination of the statements leads to
    \begin{equation*}
        -\nu < p(x_l^3; a_{l+1}) - p(x_l^3; a_l) \leq (a_{l+1} - a_l)q_{\min},
    \end{equation*}
    and thus $a_{l+1} < a_l - \nu/q_{\min}$.
    This implies the same contradiction as before and we have shown case (b).
    In conclusion, if there exists a feasible solution $(a, b, c)$,
    Method~\labelcref{alg:subroutine} converges to it.
    
    In the light of~\Cref{lem:bounds-a} the above results also imply that
    $\ub{a}_l \to \ub{a}$ for $l \to \infty$.
    If there exists no solution $(a, b, c)$, the introductory statements imply
    that $\lb a > \ub a$.
    That is, there exists some $L \in \naturals$ such that $\lb a > \ub{a}_l$
    for all $l \geq L$ and, in turn, there exists $\tilde{x} \in \locdomain$
    such that $p(\tilde{x}; \ub{a}_l) > f(\tilde{x})$.
    
    For the update of $\ub{a}_l$, note that there exist
    $x_{l-1}^1$ and $x_{l-1}^3$ that cause it.
    Then, either $p(x_{l-1}^1; \ub{a}_l) = f(x_{l-1}^1)$ or
    $p(x_{l-1}^3; \ub{a}_l) = f(x_{l-1}^3) - \eps$.
    In conclusion, to mitigate the violation in~$\tilde{x}$, there needs to be
    a reduction in parameter $\lb{a}_l$ in the next iteration,
    whereas this directly causes the respective equality to be violated again.
    Hence, the method has found a pair $\lb{a}_l < \ub{a}$ 
    and exits the while-loop.
\end{proof}

\subsection{Proof of \Cref{thm:correctness-outer-loop}}
\label{subsec:proof-outer-loop}

\thmtwo*

\begin{proof}
    Consider a local interval $\domain_k = \locdomain = [\lb t, \ub t]$
    with $\abs{\locdomain} = \Delta$.
    We show that there exists a particular $p$ such that
    \begin{itemize}
        \item[a)] $p(x) \geq f(x) - \eps$ on $\locdomain$,
        \item[b)] $p(x) \leq f(x)$ on $\locdomain$, and
        \item[c)] $p(x) \leq f(x)$ on $\domain\setminus\locdomain$.
    \end{itemize}
    For case a) we make use of~\Cref{lem:bounds-a}, whereas for
    cases b) and c) the approach is straightforward.
    Throughout the following section denote the mid point of $\locdomain$
    as $\tilde{t} = (\ub t + \lb t)/2$.
    
    Case a).
    Without loss of generality, consider $x \in \mathrm{int}(\locdomain)$.
    This implies ${(x - \lb t)(x - \ub t) < 0}$.
    
    If we can show that $a_k = -4L/\Delta \leq B(x)$, 
    we have proven case a) by \Cref{lem:bounds-a}.
    We make a distinction and consider $x \in (\lb t, \tilde{t}]$ first.
    Then, the distance of $\ub t$ to $x$ is at most $\Delta/2$.
    Formally, this is $\ub t - x \geq \Delta/2$ or equivalently 
    $x - \ub t \leq -\Delta/2$.
    
    By the Lipschitz continuity of $f$, it is
    \begin{equation*}
        -L\abs{x - y} \leq -\abs{f(x) - f(y)} \leq f(x) - f(y) \leq 
        \abs{f(x) - f(y)} \leq L\abs{x - y},
    \end{equation*}
    for $y \in \locdomain$.
    Combining it with ${(x - \lb t)(x - \ub t) < 0}$ and respecting 
    ${x - \ub t < 0}$ we can now estimate
    \begin{align*}
        B(x) &= \frac{f(x) - f(\lb t)}{(x - \lb t)(x - \ub t)} 
        - \frac{f(\ub t) - f(\lb t)}{(\ub t - \lb t)(x - \ub t)} \\
        & \geq \frac{L\abs{x - \lb t}}{(x - \lb t)(x - \ub t)}
        - \frac{-L\abs{\ub t - \lb t}}{(\ub t - \lb t)(x - \ub t)} \\
        & = \frac{2L}{x - \ub t} \geq - \frac{4L}{\Delta} = a_k,
    \end{align*}
    This was to show and concludes the first part of case a).
    
    It remains to consider $x \in [\tilde{t}, \ub t)$.
    Analogously, we derive $\lb t - x \leq -\Delta/2$.
    Now, we recall that $B(x)$ has an equivalent reformulation,
    as stated in~\Cref{rem:initial-a-conditions}, that says
    \begin{equation*}
        B(x) = \frac{f(x) - f(\ub t)}{(x - \lb t)(x - \ub t)} 
        - \frac{f(\ub t) - f(\lb t)}{(\ub t - \lb t)(x - \lb t)}.
    \end{equation*}
    Following the previous strategy, we get
    \begin{align*}
        B(x) &\geq \frac{L \abs{x - \ub t}}{(x - \lb t)(x - \ub t)}
        - \frac{L\abs{\ub t - \lb t}}{(\ub t - \lb t)(x - \lb t)} \\
        & \geq -\frac{L}{x - \lb t} - \frac{L}{x - \lb t} 
        \geq -\frac{4L}{\Delta}= a_k.
    \end{align*}
    Note that we have used $\abs{x - \ub t} = -(x - \ub t)$.
    This concludes case a).
    
    Case b).
    We make the same distinction as in case a) and 
    consider $x \in (\lb t, \tilde{t}]$ first.
    Thus $x - \lb t \leq \Delta/2$ and $\ub t - x \leq \Delta$.
    By the Lipschitz continuity of $f$, we can derive
    \begin{equation*}
        f(\lb t) - f(x) \leq \abs{f(\lb t) - f(x)} \leq L\abs{\lb t - x} 
        \leq \frac12 L \Delta,
    \end{equation*}
    or equivalently $f(\lb t) \leq f(x) + \frac12 L \Delta$.
    Using the parameterized version of $p$ as stated
    in~\labelcref{eq:parameterized-parabola}, this gives
    \begin{align*}
        p(x; a_k) &= -\frac{4L}{\Delta}(x - \lb t)(x - \ub t) + 
        (x - \lb t)\frac{f(\ub t) - f(\lb t)}{\ub t - \lb t} + f(\lb t) - \eps\\
        & \leq (x - \lb t)\left(L + \frac{4L}{\Delta}( \ub t - x)\right)
        + f(\lb t) - \eps \\
        & \leq \frac12 \Delta \left(L + \frac{4L}{\Delta} \Delta\right) 
        + \frac{1}{2}L\Delta + f(x) - \eps \\
        & \leq \frac52 L\Delta + \frac12 L \Delta + f(x) - \eps 
        = f(x) + 3 L\Delta - \eps \leq f(x),
    \end{align*}
    where the last inequality follows from the assumption on~$\Delta$.
    
    Considering $x \in [\tilde{t}, \ub t)$, we can make an analogous derivation.
    That is, we have $\ub t - x \leq \Delta/2$ and $x - \lb t \leq \Delta$,
    as well as $f(\ub t) \leq f(x) + \frac12 \Delta$.
    Combining this, we get
    \begin{align*}
        p(x; a_k) &= -\frac{4L}{\Delta}(x - \lb t)(x - \ub t) + 
        (x - \ub t)\frac{f(\ub t) - f(\lb t)}{\ub t - \lb t} + f(\ub t) - \eps\\
        & \leq (\ub t - x)\left(L + \frac{4L}{\Delta}(x - \lb t)\right) 
        + f(\ub t) - \eps \\
        & \leq \frac12\Delta \left(L + \frac{4L}{\Delta} \Delta\right) 
        + \frac{1}{2}L\Delta + f(x) - \eps \\
        & \leq \frac52 L\Delta + \frac12 L \Delta + f(x) - \eps
        \leq f(x) + 3L \Delta - \eps = f(x),
    \end{align*}
    taking the other representation of parameterized $p$
    in~\labelcref{eq:parameterized-parabola} into account.
    Note that the first step uses
    $f(\ub t) - f(\lb t) \geq -L \abs{\ub t - \lb t}$, but the sign is turned
    due to $x - \ub t \leq 0$.
    This concludes case b).
    
    Case c).
    First, we consider $x \in \domain$ with $x \geq \ub t$.
    Since $p(\ub t) = f(\ub t) - \eps < f(\ub t)$, 
    it is sufficient to show that $p' \leq f'$ for all such $x$.
    In other terms, $p$ strictly underestimates $f$ at $\ub t$ and 
    decreases faster than $f$ from $\ub t$ on.
    
    Let's investigate the derivative of $p$ at $\ub t$, which is
    \begin{equation*}
        p'(\ub t; a_k) = -\frac{4L}{\Delta}(2 \ub t - (\ub t + \lb t))
        + \frac{f(\ub t ) - f(\lb t)}{\ub t - \lb t} 
        \leq -4L \frac{\ub t - \lb t}{\Delta} +  L = -3L.
    \end{equation*}
    Since $a_k < 0$, $p$ is strictly concave and thus
    \begin{equation*}
        p'(x; a_k) \leq -3L < -L \leq f'(x),
    \end{equation*}
    for all $x \geq \ub t$.
    
    Second, consider $x \in \domain$ with $x \leq \lb t$.
    Analogously to above we need to show $p' \geq f'$ for all such $x$.
    Inspecting the derivative at $\lb t$ gives
    \begin{equation*}
        p'(\lb t; a_k) = -\frac{4L}{\Delta}(2 \lb t - (\ub t + \lb t))
        + \frac{f(\ub t) - f(\lb t)}{\ub t - \lb t} 
        \geq 4L \frac{\ub t - \lb t}{\Delta} -  L = 3L.
    \end{equation*}
    Again, due to strict concavity, we have
    \begin{equation*}
        p(x; a_k) \geq 3L > L \geq f'(x),
    \end{equation*}
    for all $x \leq \lb t$.
    This concludes case c).
    
\end{proof}

\newpage
\section{Dual Gap Improvements}
\label{sec:dual-gap-improvements}

As mentioned in~\Cref{sec:relaxation-comparison}, there exist instances for
which one relaxation method resulted in improved dual bounds.
We compare to the dual bounds achieved during the base run and the ones
reported on the MINLPLib~\cite{bussieck2003minlplib} webpage.

In~\Cref{tab:dual-gap-improvements} we list all instances, where at least one
improvement was detected and give the respective relaxation type and tolerance.
If the same type has achieved the same dual bound for several tolerances $\eps$,
we report the result for its largest value.
The column ``sense'' reports on the objective sense in order to classify the
dual bounds correctly.
The gaps are computed with respect to the reported primal value on the webpage
and -- for the relaxations -- the dual value achieved when solving a relaxation.
The term ``rep.'' stands for ``reported'' and refers to the value from the
reference value from the webpage.

We limit our interpretation to the ``lnts'' and the ``polygon'' instances, as
they show the most significant improvements.
For all those, \ac{para} was able to construct relaxation that lead to dual
bounds which can reduce the original gap to nearly zero or by a factor of
up to~10, respectively.
Once again, this strengthens our observation that \ac{para} seems to be quite
efficient when tackling (co)sine constraint functions, since all those
are nonlinear programs with quite an extensive number of such constraints.

\begin{table}
\caption{Dual and Gap Improvements by Relaxations}
\label{tab:dual-gap-improvements}
\resizebox{\textwidth}{!}{%
\begin{tabular}{lllcrrrrr}
\toprule
instance & sense & type & $\eps$ & time [\unit{\second}] & dual & rep.\,dual & gap [\%] & rep.\,gap [\%] \\
\midrule
arki0017 & min & \ac{pwl} & $10^{0}$ & limit & -9.8e2 & -1.4e3 & 70.35 & 101.24 \\
lnts100 & min & \ac{para} & $10^{-4}$ & 161 & 5.5e-1 & 5.0e-1 & 0.00 & 0.94 \\
lnts200 & min & \ac{para} & $10^{-4}$ & 505 & 5.5e-1 & 5.0e-1 & 0.00 & 0.98 \\
lnts400 & min & \ac{para} & $10^{-3}$ & 576 & 5.5e-1 & 4.5e-1 & 0.01 & 1.89 \\
lnts50 & min & \ac{para} & $10^{-4}$ & 39 & 5.5e-1 & 5.1e-1 & 0.00 & 0.87 \\
polygon100 & min & \ac{para} & $10^{-3}$ & limit & -3.1e0 & -2.8e1 & 29.78 & 352.32 \\
polygon25 & min & \ac{para} & $10^{-3}$ & limit & -2.2e0 & -4.3e0 & 17.89 & 44.80 \\
polygon50 & min & \ac{para} & $10^{-4}$ & limit & -3.0e0 & -1.1e1 & 28.47 & 133.76 \\
polygon75 & min & \ac{para} & $10^{-4}$ & limit & -3.1e0 & -1.9e1 & 29.17 & 234.13 \\
powerflow0039p & min & \ac{para} & $10^{-4}$ & limit & 4.0e2 & 4.0e2 & 9.90 & 9.90 \\
powerflow0300p & min & \ac{para} & $10^{-3}$ & limit & 1.2e4 & 0 & -- & -- \\
procurement1large & max & \ac{pwl} & $10^{-2}$ & limit & 7.1e3 & 1.8e4 & 28.66 & 36.97 \\
procurement1mot & max & \ac{pwl} & $10^{-1}$ & limit & 4.6e2 & 1.4e3 & 25.74 & 36.89 \\
\bottomrule
\end{tabular}%
}
\end{table}

\end{document}